\documentclass[fleqn,11pt,oneside]{article}

\usepackage{soul}
\usepackage{hyperref}
\usepackage[final]{graphicx}
\usepackage{bbm}  %
\usepackage{bm}  %
\usepackage{amsmath,amsthm,amssymb,amscd}
\usepackage{caption,subcaption}
\usepackage{booktabs,multirow}
\usepackage{setspace} 
\usepackage[top=1.1in, bottom=1.1in, left=1.6in, right=1.1in]{geometry}
\usepackage[final]{showlabels}
\usepackage{dashbox}
\usepackage{microtype}
\usepackage[medium]{titlesec}
\usepackage{srcltx}

\newtheorem{theorem}{Theorem}[section]

\newtheorem{corollary}[theorem]{Corollary}
\newtheorem{definition1}{Definition}[section]
\newtheorem{observe}{Observation}[section]
\newtheorem{remark1}[observe]{Remark}
\newtheorem{example1}{Example}[section]
\newtheorem{aside1}[observe]{Aside}

\newenvironment{definition}[1][]{\begin{definition1}[#1] \rm}{\end{definition1}}
\newenvironment{observation}{\begin{observe} \rm}{\end{observe}}
\newenvironment{remark}{\begin{remark1} \rm}{\end{remark1}}

\def\qed{\hfill$\blacksquare$\\} \renewenvironment{proof}{\noindent {\bf 
Proof.}}{\qed}

\newif\ifshowboxes \showboxestrue

\providecommand{\e}[1]{\ensuremath{\times 10^{#1}}}

\newcommand{\sgn}{\mathop{\mathrm{sgn}}}

\newcommand{\inner}[2]{\ensuremath{ {\langle #1,#2 \rangle} }}

\renewcommand{\d}{\,\mathrm{d}}

\newcommand{\ceil}[1]{\ensuremath{ {\lceil #1 \rceil} }}

\newcommand{\norm}[1]{\ensuremath{ {\lVert #1 \rVert} }}
\newcommand{\bnorm}[1]{\ensuremath{ {\bigl\lVert #1 \bigr\rVert} }}

\newcommand{\abs}[1]{\ensuremath{ {\lvert #1 \rvert} }}
\newcommand{\babs}[1]{\ensuremath{ {\bigl\lvert #1 \bigr\rvert} }}
\newcommand{\Babs}[1]{\ensuremath{ {\Bigl\lvert #1 \Bigr\rvert} }}
\newcommand{\bbabs}[1]{\ensuremath{ {\biggl\lvert #1 \biggr\rvert} }}

\def\R{\mathbbm{R}}
\def\N{\mathbbm{N}}
\def\1{\mathbbm{1}}
\def\ft{\mathcal{F}}
\def\nt{\mathcal{N}}
\def\st{\mathcal{S}}

\def\O{\mathcal{O}}

\usepackage{color} %
\usepackage{xcolor}
\usepackage{xspace}
\renewcommand{\tilde}{\widetilde}
\newcommand{\tDelta}{\widetilde\Delta}
\newcommand{\bDelta}{\overline\Delta}

\newcommand{\tri}{\Delta^1}
\newcommand{\origin}{O}
\newcommand{\target}{v}
\newcommand{\argmin}{\operatornamewithlimits{argmin}}
\newcommand{\nordn}{N_n}
\newcommand{\nords}{N_s}
\newcommand{\nordl}{N_l}
\newcommand{\nordf}{N_f}
\newcommand{\nordg}{N_g}
\newcommand{\nsubl}{s_l}
\newcommand{\nsubn}{s_n}
\newcommand{\ncor}{c}

\def\P{\mathcal{P}}

\begin{document}

\begin{center}
    \begin{minipage}[t]{6.0in}
      The accurate and efficient evaluation of potentials is of great
      importance for the numerical solution of partial differential
      equations. When the integration domain of the potential is irregular
      and is discretized by an unstructured mesh, the function spaces of
      near field and self-interactions are non-compact, and, thus, their
      computations cannot be easily accelerated. In this paper, we propose
      three novel and complementary techniques for accelerating the
      evaluation of potentials over unstructured meshes.  Firstly, we
      rigorously characterize the geometry of the near field, and show that
      this analysis can be used to eliminate all the unnecessary near field
      interaction computations.  Secondly, as the near field can be made
      arbitrarily small by increasing the order of the far field quadrature
      rule, the expensive near field interaction computation can be
      efficiently offloaded onto the FMM-based far field interaction
      computation, which leverages the computational efficiency of highly
      optimized parallel FMM libraries. Finally, we show that a separate
      interpolation mesh that is staggered to the quadrature mesh
      dramatically reduces the cost of constructing the interpolants.
      Besides these contributions, we present a robust and extensible
      framework for the evaluation and interpolation of 2-D volume
      potentials over complicated geometries. We demonstrate the
      effectiveness of the techniques with several numerical experiments.

 \vspace{ 0.15in}

   \vspace{ -100.0in}
 
 \thispagestyle{empty}

   \end{minipage}
 \end{center}
 
 \vspace{ 3.10in}
 \vspace{ 0.80in}
 
 \begin{center}
   \begin{minipage}[t]{4.4in}
     \begin{center}

\textbf{Accelerating potential evaluation over unstructured meshes in two dimensions}
\\
   \vspace{ 0.30in}
 
 Zewen Shen$\mbox{}^{\dagger\, \star}$ and
 Kirill Serkh$\mbox{}^{\ddagger\, \diamond}$  \\
               June 16, 2022 
 
     \end{center}
   \vspace{ -50.0in}
   \end{minipage}
 \end{center}
 
 \vspace{ 1.05in}

 \vfill
 
 \noindent 
 $\mbox{}^{\diamond}$  This author's work was supported in part by the NSERC
 Discovery Grants RGPIN-2020-06022 and DGECR-2020-00356.
 \\

 \vspace{2mm}
 
 \noindent
 $\mbox{}^{\dagger}$ Dept.~of Computer Science, University of Toronto,
 Toronto, ON M5S 2E4\\
 \noindent
 $\mbox{}^{\ddagger}$ Dept.~of Math. and Computer Science, University of Toronto,
 Toronto, ON M5S 2E4 \\
 
 \vspace{2mm}
 \noindent 
 $\mbox{}^{\star}$  Corresponding author
 \\

 \vfill
 \eject
\tableofcontents

\section{Introduction}

Integral equation methods solve partial differential
equations (PDEs) by reformulating them as integral equations using the tools
of potential theory. The accurate and efficient evaluation of potentials
with singular or weakly-singular kernels, over curves, surfaces and volumes,
is thus of great importance for the numerical solution of PDEs.  However,
their numerical evaluation poses several difficulties. Firstly, the
potential is often expressed as a
convolution of a Green's function with a density function, and due to the
singularity of the Green's function, special integration techniques must be
employed.  Secondly, the integration domain of the potential could be
complicated, which requires it to either be embedded into a larger regular
domain, or to be resolved by adaptive meshing. Finally, the scheme for
evaluating the potential must be compatible with fast algorithms (e.g.,
the fast multipole method (FMM) or the fast Fourier transform (FFT)) for achieving
linear or quasi-linear time complexity. We note that, when the integration
domain is regular, the difficulties stated above can be easily overcome by
exploiting the translational invariance of the free-space Green's function.
More specifically, given a regular domain that is discretized by a rectangular
mesh, the dimensionality of the function spaces of near field and self-interactions is finite,
and, thus, these interactions can be efficiently tabulated, from which it
follows that the box code \cite{box1,askham} can be used to compute the potential in
linear time with a small constant. However, when the integration domain is
irregular, provided that the domain is discretized by an unstructured mesh,
the function spaces of near field and self-interactions are non-compact,
and, thus, one can no longer easily accelerate the computation by precomputations.

Existing methods for computing the potential over an irregular domain
generally fall into two categories. The first one is based on the
observation that the volume potential is the solution to an elliptic
interface problem, where the irregular domain $\Omega$ is embedded inside a
regular box (see, for example, \cite{mayo1}, for details). A finite
difference method with corrections based on knowledge of the jumps in the
solution across $\partial \Omega$ is applied, with the order of accuracy
determined by the finite difference scheme used. Additionally, the method is
compatible with the FFT when a uniform grid is used to discretize the
domain, and, thus, it can achieve a quasi-linear time complexity. We note
that, although this approach saves the trouble of meshing an irregular
domain, its order of convergence is usually low, and it is not highly
compatible with adaptive mesh refinement. Furthermore, the method
doesn't easily generalize to the surface potential case.

The methods belonging to the second category compute the potential
directly by quadrature. More specifically, the domain $\Omega$ is
discretized into an unstructured mesh, and for each target
location $x$, depending on its proximity to the mesh elements, different
quadrature schemes are used to compute the integral over each mesh element,
which leads to a spectrally accurate evaluation. Moreover, the computation
of the far field interactions (i.e., integrals over mesh elements that are
far away from $x$) can be accelerated by the FMM,
which results in linear total time complexity. We note that research into
methods of this type has mostly focused on the surface potential case,
despite the fact that the algorithms for computing the surface potential
have great similarities with the ones for computing the 2-D volume
potential. 
We refer the readers to \cite{anderson} for a thorough literature review of
methods belonging to these two categories.

Despite the advantages of the direct approach of computing the potential by
quadrature (i.e., spectral accuracy and linear time complexity), the 
actual constants in time complexity are usually large, due to the cost
associated with the near field and self-interaction
computations, which cannot be efficiently precomputed over an unstructured
mesh. In \cite{anderson}, the authors propose a
remedy to this issue. Given an irregular domain $\Omega$, they embed a
rectangular mesh inside a large regular subdomain of $\Omega$, and fill the
rest of the domain with unstructured meshes that conform to the curved
boundary. It follows that a large proportion of the evaluations can be
accelerated by the box code, while respecting the true geometry of the
domain to retain spectral accuracy. However, as the authors of
\cite{anderson} point out, the computation of near field and
self-interactions over the unstructured mesh forms the majority of the costs
in their algorithm, even when the unstructured mesh elements only make up a
small proportion of the total mesh elements. Furthermore, such a remedy is
not applicable to surface potential evaluation. Therefore, we regard
potential evaluation over unstructured meshes as a problem of great
importance in integral equation methods.

In this paper, we propose the following novel and complementary
techniques for accelerating the near and self-interaction computations over
an unstructured mesh.  Below, we briefly describe the techniques.
\begin{enumerate}
\item In the classic literature, the near field is typically
approximated by a ball or a triangle. We observe that this often leads to an overestimation
of the true near field, especially when the order of the quadrature or
the error tolerance is high, which leads to substantial unnecessary and expensive
near field interaction computations. Thus, we rigorously 
characterize the geometry of the near field, which allows us to dramatically
reduce the number of the required near field interaction computations.
Specifically, we show that the near field is approximately equal to the union
of several Bernstein ellipses.  In addition, this technique provides error
control functionality to the far and near field interaction computations.
For example, this technique allows for a precise determination of the number
of required subdivisions of the element domain when the near field
interactions are computed adaptively, which avoids the possibility both of
oversampling and of undersampling. 
\item Since our analysis shows that the near field can be made arbitrarily
small by increasing the order of the far field quadrature rule, we observe
that one can efficiently offload the near field interaction computation onto
the FMM-based far field interaction computation. This trade leverages the
computational efficiency of highly optimized parallel FMM libraries, and
reduces the cost of the much more expensive and unstructured near field
interaction computation.  Furthermore, this offloading technique is 
one of the few applications we are aware of that requires the use of
extremely high-order (say, $50$th order) quadrature rules in high
dimensions.
\item When one interpolates the potential, we observe that the most commonly
used arrangement of the quadrature nodes and interpolation nodes, where they
are both placed over a single mesh, leads to an artificially large near
field and self-interaction computation cost. If the quadrature and
interpolation nodes are instead placed over two separate meshes that are
staggered to one another, the number of interpolation nodes at which the
near field and self-interactions are costly to evaluate is reduced
dramatically. 
\end{enumerate}
We note that the near field geometry analysis for the 1-D layer potential
computed by the trapezoidal rule and the Gauss-Legendre rule has been
rigorously carried out in \cite{barnett} and in \cite{klin,klin1}, respectively, and
recently, the authors of \cite{klin2} characterize the near field geometry of the
surface potential in $\R^3\setminus S$, where $S$ is the surface over which
the potential is generated. The near field geometry analysis for volume
potentials has not been carried out previous to this paper.  Furthermore, we
point out that the near field geometry analysis is much more powerful in
high dimensions (i.e., 2-D volumes, 3-D volumes, surfaces), since the
intersection of the near field with the domain over which the potential is
generated becomes non-negligible in these situations, and thus, the
computation of the near field interactions becomes a bottleneck when solving
integral equations over these domains. We also note that the idea of
increasing the order of the far field quadrature rule to reduce the amount
of near field interaction computation appears, for example, in
\cite{leslie,anderson}. However, it is presented as a heuristic. In fact, without
characterizing the geometry of the near field precisely, such an idea cannot
be optimally carried out. As we show in this paper, the near field is
approximately equal to the union of several Bernstein ellipses, and when the order of
the far field quadrature is high, their area becomes vanishingly small.
Thus, a naive estimation of the near field (by a ball \cite{leslie,bremer1}
or a triangle \cite{anderson}) is seen to be poor, and a large proportion of
the near field interaction computations are unnecessary. In addition to this,
when a naive estimation of the near field is used, one has to overestimate
the size of the near field to improve the robustness of the algorithm, which
further increases the unnecessary cost.

One would generally expect that a reduction in the area of the near field
leads to a reduction in the cost of the computation of near field
interactions. However, this turns out to be false, in the common case where
the target points are chosen to be interpolation nodes over the same
triangle mesh as the one that is used for quadrature. The discretization
nodes tend to cluster near the edges of the mesh elements, which means that
a smaller near field does not necessarily result in a commensurate reduction
in the computational cost. With the use of a separate interpolation mesh
that is staggered to the quadrature mesh, the reduction in the near field
interaction computation cost becomes proportional to the reduction in the
area of the near field. The techniques that we present are thus
complementary.

In this paper, for simplicity, we only consider the evaluation of the 2-D
Newtonian potential over an irregular domain $\Omega$,
  \begin{align}
u(x)=\iint_{\Omega} \frac{1}{2\pi}\log\norm{x-y}f(y)\d A_y.
  \end{align}
We note that our techniques can be easily generalized to kernels of other
types and, also, with some additional work, to surface and 3-D
volume potential evaluations. 

Besides the general contributions to potential evaluation problems that we
describe above, we make the following contributions that are specific to
the 2-D volume potential evaluation problem.
\begin{itemize}
\item We describe a robust and extensible framework for evaluating and
interpolating 2-D volume potentials over complicated geometries, without
the need for extensive precomputation. Our presentation makes minimal use of
specialized quadrature rules, although our framework is compatible with
their use.
\item Although a well-conditioned formula for computing self-interactions has
already appeared, for example, in \cite{zhuthesis,anderson}, its standard derivation 
is somewhat overcomplicated and the necessity of the formula is not
fully motivated. We instead present a short and elementary derivation of
the same formula from the first principles, and explain why it is needed.
\item We provide a description of a full and robust pipeline of
the geometric algorithms that are required for the volume potential
evaluation, e.g., a meshing algorithm, a nearby element query algorithm,
etc.
\end{itemize}

We conduct several numerical experiments to demonstrate the performance of
the techniques for accelerating the computation over an unstructured mesh.
We also report the overall performance of our volume potential evaluation
algorithm.

\section{Mathematical preliminaries}

\subsection{Koornwinder polynomials}
  \label{sec:koorn}
The Koornwinder polynomials, denoted by $K_{nm}:\tri\to \R$, are defined by
\begin{align}
\hspace*{-1em}K_{mn}(x,y)=c_{mn}(1-x)^mP_{n-m}^{(2m+1,0)}(2x-1)P_m\Bigl(\frac{2y}{1-x}-1\Bigr),\quad
m\leq n,
\end{align}
where 
\begin{align}
\tri=\{(x,y)\in \R^2 : 0\leq x\leq 1,0\leq y\leq 1-x\}
\end{align}
is the standard simplex, $P_{k}^{(a,b)}$ is the Jacobi polynomial of degree
$k$ with parameters $(a, b)$, $P_m$ is the Legendre polynomial of degree
$m$, and $c_{mn}$ is the normalization constant such that
\begin{align}
\int_{\tri}\abs{K_{mn}(x,y)}^2\d x \d y = 1.
\end{align}

It is observed in \cite{koorn} that the $(N+1)(N+2)/2$ functions
\begin{align}
\bigl\{K_{mn}(x,y) : n=0,\dots,N, m=0,\dots,n\bigr\}
\end{align}
form an orthogonal basis for $\P_N$ on the standard simplex $\tri$, where
$\P_N$ denotes the space of polynomials of order at most $N$ on $\tri$.
In addition, by orthogonality, we also have that 
  \begin{align}
\int_{\tri} K_{mn}(x,y) \d x\d y = 0,
  \label{for:koorn_ortho}
  \end{align}
for any $m,n\in \N_{\geq 0}$ with $m+n>0$.

\subsection{Quadrature and interpolation over triangles}
  \label{sec:vioreanu}

On a two-dimensional domain, a quadrature rule of length $n$ is optimal if
it integrates $3n$ functions (since the total number of degrees of freedom
of the rule is $3n$), and we refer to such a rule as a generalized Gaussian
quadrature rule. In general, the efficiency of a quadrature rule is defined
to be $E=\frac{m}{(d+1)n}$, where $d$ is the dimensionality of the domain,
$n$ is the length of the quadrature rule, and $m$ is the dimensionality of
the space of functions that can be integrated exactly using that rule (see
\cite{xiao} for details).

Although the construction (or even the existence) of perfect generalized Gaussian
quadrature rules over two-dimensional domains remains an open problem,
various schemes for generating nearly-perfect ones exist, e.g.,
\cite{vior,xiao}. In this section, we describe some quadrature and
interpolation schemes for polynomials over a triangle domain.

The Vioreanu-Rokhlin rule, introduced in \cite{vior},
takes two integers $N$ and $M\geq N$ as inputs, and attempts to generate
a quadrature rule of length exactly $\dim \P_N$ that integrates all
functions in $\P_M$ over a given convex domain exactly.
The method is based on the observation that elements of the complex spectrum
of the multiplication operator restricted to $\P_N$, acting on any convex
domain, turn out to be excellent quadrature nodes for integrating all
functions in $\P_N$ over the domain. These nodes are used as an initial
guess by the Vioreanu-Rokhlin algorithm and iteratively improved (by solving
a nonlinear least-squares problem) to integrate all
functions in $\P_M$. As a result, the generated rule is generally efficient
and of high quality.  Additionally, the set of quadrature nodes can also
serve as interpolation nodes for approximating functions in $\P_N$, since
the length of the rule equals $\dim \P_N$. To get the interpolation matrix,
we invert the matrix that maps the Koornwinder polynomial expansion
coefficients to function values at the interpolation nodes. Using the
Vioreanu-Rokhlin nodes as the interpolation nodes, the condition number
of this interpolation matrix turns out to be small, which guarantees 
that the interpolation is stable.

In the situation where interpolation is not needed, one can loosen the
restriction in the Vioreanu-Rokhlin algorithm that the length of the
quadrature rule equals $\dim \P_N$ for some $N$, and this can result in a
more efficient quadrature algorithm. This fact is used by the Xiao-Gimbutas
algorithm \cite{xiao}, which is also based on solving a nonlinear
least-squares problem.  For example, the Xiao-Gimbutas rule of length 78
integrates $\P_{20}$ exactly, while the Vioreanu-Rokhlin rule needs to have
a length equal to $\dim \P_{12}=91$ to accomplish the same job.  We tabulate
the lengths and orders of some Xiao-Gimbutas rules and Vioreanu-Rokhlin
rules in Tables \ref{tab:ord_len_quad} and \ref{tab:ord_len_interp}.

\subsection{The polar tangential angle}

In this section, we describe some concepts in differential geometry that we
make use of in the sequel.

\begin{definition}
Given a unit-speed parametrized curve $\gamma(s)$, we define the polar
tangential angle $\psi(s)$ of $\gamma$ with respect to polar coordinates
centered at $\origin$ to be 
  \begin{align}
\psi(s)=\angle (\gamma'(s),\gamma(s)-\origin),
  \end{align}
where $\gamma$ is parameterized by arc length (see Figure
\ref{fig:polar_tang}).  In other words, $\psi(s)$ represents the angle
between the tangent line to the curve at $\gamma(s)$ and the ray from
$\origin$ to the point. 
\end{definition}

\begin{figure}[h]
    \centering
    \includegraphics[width=0.40\textwidth]{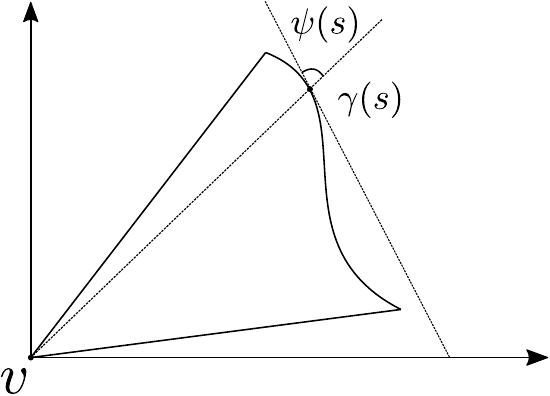}
  \caption{
      {\bf The polar tangential angle $\psi(s)$.}}
   \label{fig:polar_tang}
\end{figure}
Let the polar angle $\theta(s)$ be the angle of the point $\gamma(s)$ in
polar coordinates with respect to the origin $\origin$.  The following theorem
illustrates how to compute the derivative of the polar angle of a given
curve, given its radius and polar tangential angle. The proof of the
theorem can be found in Chapter 12 of \cite{polartan}.

\begin{theorem}
  \label{thm:thetap}
Suppose that $\gamma(s)$ is a unit-speed parametrized curve.
Suppose further that $r_0(s)$ and $\theta(s)$ are the radius and polar
angle, respectively, of $\gamma(s)$ with respect to a given polar coordinate
system. Then
\begin{align}
\theta'(s)=\frac{\sin(\psi(s))}{r_0(s)},\label{for:thetap1}
\end{align}
where $\psi(s)$ is the polar tangential angle of $\gamma$ at $\gamma(s)$.
\end{theorem}

It turns out that the derivative $\theta'(s)$ can be expressed directly in
terms of $\gamma(s)$ and $r_0(s)$.

\begin{corollary}
  \label{cor:thetap}
Suppose that $\gamma(s)$ is a unit-speed parametrized curve.
Suppose further that $r_0(s)$ and $\theta(s)$ are the radius and polar
angle, respectively, of $\gamma(s)$ with respect to a given polar coordinate
system centered at $\origin$. Then

  \begin{align}
\theta'(s)=\frac{\abs{(\gamma(s)-\origin)\times \gamma'(s)}}{r_0^2(s)}.
  \end{align}
\end{corollary}
\begin{proof}
By definition of the cross product, we have that
  \begin{align}
\abs{(\gamma(s)-\origin)\times \gamma'(s)}=r_0(s)\sin(\psi(s)).\label{for:thetap3}
  \end{align}
By combining (\ref{for:thetap1}) and (\ref{for:thetap3}), we prove the result.
\end{proof}

\subsection{Approximation of analytic functions by polynomials}
In this section, we describe several concepts in approximation theory that
we make use of in the sequel.

\begin{definition}
Given a real number $\rho>1$, the Bernstein ellipse with parameter $\rho$ is
defined to be the ellipse
  \begin{align}
\Bigl\{(z+z^{-1})/2:z= \rho e^{i\theta}, \theta\in[0,2\pi)\Bigr\}.
  \end{align}
Furthermore, we denote the interior of the Bernstein ellipse by $E_\rho$. In
a slight abuse of notation, we also refer to $E_\rho$ as the Bernstein
ellipse, when the meaning is clear.
\end{definition}

The proof of the following two theorems can be found in \cite{nick}.
\begin{theorem}
  \label{thm:bern}
Suppose $f$ is a function on $[-1,1]$ for which there exist
a sequence of polynomials $q_0,q_1,\dots$, where $q_n$ is a polynomial
of order $n$, satisfying 
  \begin{align}
\norm{f-q_n}_{L^{\infty}[-1,1]}\leq C\rho^{-n},
  \end{align}
for some integer $n\geq 0$, and some real numbers $\rho>1, C>0$. Then $f$
can be analytically continued to an analytic function in $E_\rho$.
\end{theorem}

\begin{theorem}
  \label{thm:anal_conv}
Suppose that an analytic function $f:[-1,1]\to\R$ is analytically continuable
to $E_\rho$, for some real number $\rho>1$. Suppose further that
$\abs{f(z)}$ is bounded inside $E_\rho$ by some constant $M$. Then, 
  \begin{align}
\norm{f-f_n}_{L^{\infty}[-1,1]}\leq \frac{2M\rho^{-n}}{\rho-1},
  \end{align}
for all $n\geq 0$, where $f_n$ denotes the $n$th order Chebyshev
projection of $f$, by which we mean the infinite expansion of $f$ in terms
of Chebyshev polynomials, truncated at the $n$th order (inclusive).
\end{theorem}

The following corollary is an immediate consequence of Theorem \ref{thm:anal_conv}.
\begin{corollary}
Suppose that an analytic function $f:[-1,1]\to\R$ is analytically continuable
to $E_\rho$, for some real number $\rho>3/2$. Suppose further that
$\abs{f(z)}$ is bounded inside $E_\rho$ by $1/4$. Then, 
  \begin{align}
\norm{f-f_n}_{L^{\infty}[-1,1]}\leq \rho^{-n},
  \end{align}
for all $n\geq 0$, where $f_n$ denotes the $n$th order Chebyshev
projection of $f$, by which we mean the infinite expansion of $f$
in terms of Chebyshev polynomials, truncated at the $n$th order (inclusive).
  \label{cor:anal_conv}
\end{corollary}

\section{Volume potential evaluation over complicated geometries}
  \label{sec:volpot}
In this section, we describe a numerical apparatus for computing and
interpolating the volume potential
  \begin{align}
u(\target)=\iint_{\Omega} \log(\norm{\target-y})f(y)\d A_y,\label{for:volpot}
  \end{align}
where the target location $\target\in \R^2$, and $\Omega$ is a smooth planar
domain. We first discretize $\Omega$ into a triangle mesh (see
Appendices \ref{sec:distmesh}, \ref{sec:mod_distmesh} for details), and reduce the
problem to the case where the integration region is either a triangle
$\bDelta$ or a curved element $\tDelta$ (a mesh element with three sides,
one of which is curved). Furthermore, we divide each of these two problems
into three distinct subproblems: where $\target$ is far from the element
(i.e., far field interactions), where $\target$ is close to the element (i.e.,
near field interactions), where $\target$ lies within the element (i.e.,
self-interactions).
Our goal is to present a simple and robust framework for
handling these different types of the interactions. In particular, we
minimize the use of special-purpose quadrature rules in our discussion, although
it could potentially speed up the algorithm. We note that, however, our
framework is compatible with such rules.
In addition, we describe an interpolation scheme for the volume potential
over a mesh element. In the end, we describe how to couple the algorithms
with the fast multipole method (FMM) to evaluate the volume potential
generated over $\Omega$ at any given set of target locations, with linear
time complexity.

For simplicity, we refer to both standard triangles and curved elements as
mesh elements when there is no ambiguity.

\subsection{Integration over a mesh element}
  \label{sec:int_smooth}

In this section, we introduce algorithms for generating efficient
quadrature rules for integrating smooth functions over an arbitrary mesh 
element $\Delta$ (either a triangle or a curved element).

Recall that in Section \ref{sec:vioreanu}, we describe how to obtain a
nearly-perfect Gaussian quadrature rule for integrating polynomials over the
standard simplex 
  \begin{align}
\tri=\{(x,y)\in \R : 0\leq x\leq 1,0\leq y\leq 1-x\}. 
  \end{align}
Clearly, provided that an invertible and well-conditioned mapping
$\rho:\tri\to\Delta$ is given, the generation of a quadrature rule over
$\Delta$ is simple, and is described in the following theorem.

\begin{theorem}
  \label{thm:smooth_int_jac}
Given the standard simplex $\tri$ and a mesh element $\Delta$, suppose that
$\{(x_i,y_i),w_i\}_{i}$ is a set of quadrature nodes and weights
that integrates $\{K_{mn}\}_{m,n}$ exactly over $\tri$, and $\rho$ is
an invertible and well-conditioned mapping from $\tri$ to $\Delta$. 
Then,  
  \begin{align}
\bigl\{\rho(x_i,y_i),w_i\cdot \abs{J_\rho(x_i,y_i)}\bigr\}_{i}
  \end{align}
is the set of quadrature nodes and weights that integrates
$\{K_{mn}\circ \rho^{-1}\}_{m,n}$ exactly over the mesh element
$\Delta$, where $J_\rho(x)$ denotes the Jacobian determinant of $\rho$ at
$x$.
\end{theorem}
\begin{proof}
The theorem follows directly from a change of variables.
\end{proof}

Therefore, the problem of evaluating an integral over a mesh element reduces
to the computation of an invertible and well-conditioned mapping $\rho$ from
$\tri$ to $\Delta$. When $\Delta$ is a triangle, there exists an affine
transformation from $\tri$ to $\Delta$, so the mapping $\rho$ can be 
constructed easily, and
$\abs{J_\rho(x_i,y_i)}=\text{area}(\Delta)/\text{area}(\tri)$. When
$\Delta$ is a curved element, the blending function method
\cite{blending1,blending2,blending3} provides an elegant solution to the
problem, as is mentioned in \cite{anderson}. Below, we describe the method.

Let $\gamma:[0,L]\to\R^2$ be the unit-speed parametrization of the curved
side of a curved element $\tDelta$, and let $\origin:=(x_0,y_0)$ denote the vertex
opposite to the curved side. Then for any point $(\xi,\eta)\in \tri$, define 
  \begin{align}
\hspace*{-0em}\rho(\xi,\eta) :=&\ (1-\xi-\eta)\cdot\gamma(L) +
\xi\cdot\gamma(0)+\eta\cdot\origin\notag \\+
&\frac{1-\xi-\eta}{1-\xi}\Bigl(\gamma\bigl(L(1-\xi)\bigr) -
(1-\xi)\cdot\gamma(L)-\xi\cdot\gamma(0)\Bigr).
  \label{for:blend}
  \end{align}
It can be shown that $\rho$ is an invertible and well-conditioned mapping
from $\tri$ to $\tDelta$, and its Jacobian $J_\rho$ can be
computed by a straightforward calculation. 

\begin{remark}
When the curved side $\gamma$ is a straight line segment (i.e., the
curved element is a triangle), the mapping $\rho$ generated by the blending
function method degenerates into an affine mapping, which implies that the
discussion in this section can be unified under the blending function method
framework. For clarity and computational efficiency purposes, the discussions
of the two cases are separated.
\end{remark}

\begin{observation}
In the case where the curved side of a curved element is not well-resolved
by the mesh, the Jacobian of the blending function $\rho$ becomes non-smooth,
which requires a relatively large number of degrees of freedom to approximate.
Thus, when the requirement for high accuracy is rigid, it is important to refine
the mesh around the region where the geometry of the boundary is
complicated, or use a quadrature rule of order higher than the one for
triangle elements. 
  \label{obs:nonsm_bjac}
\end{observation}

\subsection{Far field interactions}
  \label{sec:far_int}
In this section, we introduce far field quadrature rules for computing the
volume potential (\ref{for:volpot}) when the integration domain $\Omega$ is
a mesh element, and the target location $\target$ is in the far field of the
domain $\Omega$ (formally defined at the end of this section). In this case,
the integrand of (\ref{for:volpot}) is a smooth function, so the quadrature
rules described in Section \ref{sec:int_smooth} are applicable.  

Given an arbitrary mesh element $\Delta$ (either a triangle or a curved
element), and an invertible and well-conditioned mapping
$\rho$ from the standard simplex $\tri$ to $\Delta$ (see Section
\ref{sec:int_smooth}), we have that, by Theorem \ref{thm:smooth_int_jac},
  \begin{align}
\hspace*{-2em}\iint_{\Delta} \log(\norm{\target-y})f(y)\d A_y\approx \sum_{i=1}^{N}
w_i\cdot \abs{J_\rho(\tilde y_i)}\cdot\log(\norm{\target-\rho(\tilde
y_i)})\cdot f(\rho(\tilde y_i)),
  \end{align}
where $\tilde y_i=(x_i,y_i)$, and
$\{(x_i,y_i),w_i\}_{i=1,2,\dots,N}$
is a set of quadrature rule over $\tri$ (see Section \ref{sec:vioreanu}).

\begin{remark}
In general, it is preferable to have the mesh elements to be almost equilateral,
in which case the basis function has the same amount of expressibility 
in the $x$- and $y$-directions.
\end{remark}

We now give a precise definition of the set of target points in the far field and
near field of a mesh element. Additionally, we provide a dual definition
describing the set of mesh elements in the far field and near field of a
target point.

\begin{definition}
  \label{def:far_near_field}
Given a mesh element $\Delta$ (either a triangle or a curved element), a set of
possible densities $S$, a far field quadrature rule of order $N$,
and an error tolerance $\varepsilon$, the far field of the mesh element,
denoted by $\ft_{\Delta}$, is the set of targets in $\R^2$ where the volume
potential generated by a density $f\in S$ over the element can be computed
up to precision $\varepsilon$ by applying the given far field quadrature
rule. The near field of the mesh element, denoted by $\nt_{\Delta}$, is the
set of locations that do not belong to either the far field of the mesh
element, or the mesh element itself. 
\end{definition}

\begin{definition}
  \label{def:nearby_mesh_type}
Given a target location $x$ and a set of mesh element $\mathfrak{F}$, we define  
$\st(x)\in\mathfrak{F}$ to be the element that $x$ lies within. Furthermore,
given a far field quadrature rule of order $N$, a set of
possible densities, and an error tolerance
$\varepsilon$, we define $\nt(x)\subseteq \mathfrak{F}\setminus{\st(x)}$ to be the
set of all elements whose near fields contain $x$ (see Definition
\ref{def:far_near_field}), and define $\ft(x):=\mathfrak{F}\setminus(\st(x)\cup
\nt(x))$. With a slight abuse of notation, we refer to $\ft(x)$ and $\nt(x)$
as the far field and near field of $x$, respectively.
\end{definition}

\subsection{Near field interactions}
  \label{sec:near_int}
In this section, we introduce an algorithm for computing the volume
potential (\ref{for:volpot}) when the integration domain $\Omega$ is a mesh
element, and the target location $\target$ is in the near field of $\Omega$
(see Definition \ref{def:far_near_field} for the definition of the near field).
In this case, the integrand of (\ref{for:volpot}) is nearly-singular, so the
far field quadrature rule described in Section \ref{sec:far_int} will not
achieve sufficient accuracy. Below, we describe an adaptive algorithm for
resolving the near-singularity in the integrand. 

\begin{figure}[h]
    \centering
    \includegraphics[width=0.25\textwidth]{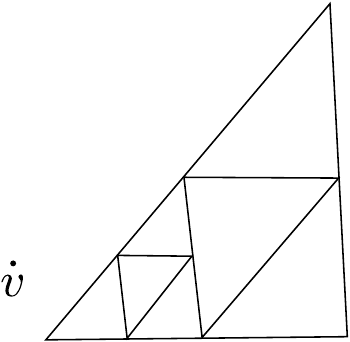}
  \caption{
      {\bf A subdivided triangle mesh element during the computation
      of near field interactions}.  }
   \label{fig:tri_subdiv}
\end{figure}

For clarify, we denote the integration domain (i.e., a mesh element) by $\Delta$.
Furthermore, suppose that $\rho$ is an invertible and well-conditioned
mapping from the standard simplex $\tri$ to $\Delta$. Firstly, we recursively
subdivide $\tri$, as is shown in Figure \ref{fig:tri_subdiv}, until all the
sub-simplexes are mapped to sub-mesh elements that belong to $\ft(\target)$
(see Definition \ref{def:nearby_mesh_type}). More formally, if we let
  \begin{align}
  \{\tri_i: i=1,2,\dots,N\}
  \end{align}
denote the set of such sub-simplexes, we have that $\bigl\{\tri_i\bigr\}_i$
are disjoint, and $\cup_{i=1}^N \tri_i=\tri$.  Therefore,
  \begin{align}
\hspace*{-2em}\iint_{\Delta} \log(\norm{\target-y})f(y)\d A_y=
\iint_{\tri} \log(\norm{\target-\rho(y)})f(\rho(y))\abs{J_\rho(y)}\d A_y\notag\\
\hspace*{-2em}=\sum_{i=1}^N
\iint_{\tri_i} \log(\norm{\target-\rho(y)})f(\rho(y))\abs{J_\rho(y)}\d A_y.
\label{for:sum_near}
  \end{align}
Furthermore, since the target $\target$ is now located in the far field of
$\rho(\tri_i)$ for all $i$, the integral
  \begin{align}
\iint_{\tri_i} \log(\norm{\target-\rho(y)})f(\rho(y))\abs{J_\rho(y)}\d A_y
  \end{align}
can be computed both accurately and efficiently using the quadrature rule
over triangles (see Section \ref{sec:vioreanu}). Thus, the near-field 
interactions over a triangle can be computed accurately via (\ref{for:sum_near}),
despite the near-singularity in the integrand.

The determination of the far field $\ft(\target)$ turns out to be important
for the speed of the near interaction computation. We postpone the
discussion to Section \ref{sec:nf}.

\begin{observation}
The order of the quadrature rule used in the near interaction computation is
unrelated to the order of the far field quadrature rule. In
fact, the order of the far field quadrature rule only affects whether or 
not a target is in the far field of some mesh element. Since the near field
quadrature rule is used in a recursively subdivided partition, its order
should be modest. 
\end{observation}

\begin{remark}
It is suggested in \cite{leslie} that one should store the function values of
the density at the quadrature nodes of the subdivided triangles used
in the computation of near field interactions, such that these values can be
reused in future near-field computations. This approach can lead to a
large improvement in the performance, since it avoids lots of recurring
evaluations of the density function. In large-scale parallel applications,
this runs the risk of turning a compute-bound task into a memory-bound task,
which can cause a degradation in performance. We note that this technique
is not employed in our implementation.
\end{remark}

\subsection{Self-interactions}
  \label{sec:self_int}

In this section, we introduce an algorithm for computing the volume
potential (\ref{for:volpot}) when the integration domain $\Omega$ is a
mesh element, and the target $\target$ lies within the
integration domain.  In this case, the integrand of (\ref{for:volpot}) has a
singularity at $y=\target$, so some special treatment is required to resolve it.

Suppose that the integration domain $\Omega$ is a triangle with
vertices $A,B,C\in \R^2$, and $\target:=(x_0,y_0)\in\R^2$ is a target that lies
within $\Omega$. We then rewrite the integral (\ref{for:volpot}) as
  \begin{align}
\iint_{\Omega} \log(\norm{\target-y})f(y)\d A_y=\sum_{i=1}^3\iint_{\bDelta_i}
\log(\norm{\target-y})f(y)\d A_y,
  \end{align}
where $\bDelta_i$ is a subtriangle with vertices given by $\target$ and two
of $A,B,C$. When $\Omega$ is a curved element,
one of $\{\bDelta_i\}_{i=1,2,3}$ also becomes a curved element
$\tDelta_i$ with $\target$ being the vertex opposite to the curved
side, while the other two remain subtriangles (see Figure~\ref{fig:self_subdiv}).
Therefore, the problem of computing self-interactions reduces to the problem
of computing 
  \begin{align}
\iint_{\Delta}\log(\norm{\target-y})f(y)\d A_y,\label{for:volpot_self}
  \end{align}
where the element $\Delta$ may or may not have a curved side, and $\target$
is a vertex of $\Delta$ (specifically, the vertex opposite to the curved side,
if a curved side exists). With a slight abuse of
notation, we refer to the side opposite to the vertex $\target$ as the
curved side, despite that it could be a straight line.

\begin{figure}[h]
    \centering
    \includegraphics[width=0.40\textwidth]{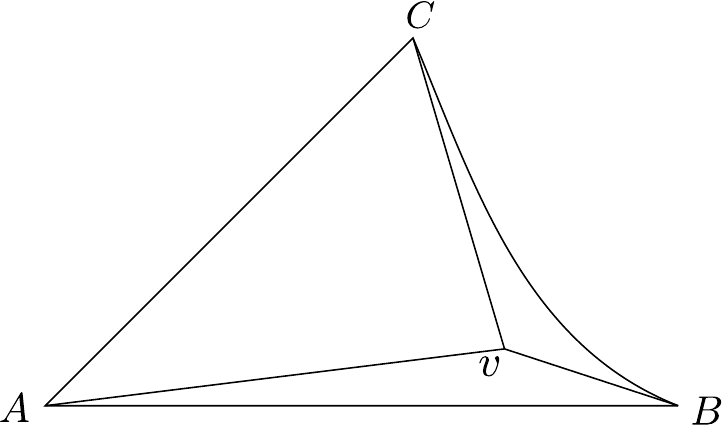}
  \caption{
      {\bf A subdivided curved element during the computation of
      self-interactions}. }
   \label{fig:self_subdiv}
\end{figure}

To evaluate the self-interaction, we first simplify the integrand in
(\ref{for:volpot_self}) by rewriting the double integral in a different
coordinate system. There are many possible coordinates to use, and we find 
that the so-called radius-arc length coordinates have many desirable
properties, which we make use of later on.

Let $\gamma:[0,L]\to\R^2$ be the unit-speed parametrization of the
curved side of $\Delta$. We write
  \begin{align}
\hspace{-2em}\Delta=\bigl\{(x_0+r\cos(\theta(s)),y_0+r\sin(\theta(s))):0\leq s\leq
L,0\leq r\leq r_0(s)\bigr\},\label{for:tdelta_domain}
  \end{align}
where $r_0(s)$ and $\theta(s)$ are the radius and polar angle of the curved
side $\gamma$ of $\Delta$ with respect to polar coordinates centered at
$\target:=(x_0,y_0)$, and $s\in [0,L]$ represents the arc length parameter
of the curved boundary. We also denote the inverse of $\theta(s)$ by
$s(\theta)$.

\begin{theorem}
  \label{thm:tensor_smooth_new}
Suppose that $\Delta$ is a mesh element as defined in (\ref{for:tdelta_domain}),
and $f:\Delta\to \R$ is a function, such that $\log(\norm{\target-y})\cdot f(y)$
is integrable over $\Delta$. Supposing further that $\theta(s)$ is strictly
monotone, we have that
  \begin{align}
\hspace*{-6.5em}\iint_{\Delta}\log(\norm{\target-y})f(y)\d A_y
=\int_{0}^{L}\Bigl(\int_0^1 \log(r_0(s)\tilde r)\cdot f\bigl(y(\tilde r,s)\bigr)
\cdot \tilde r \d \tilde r \Bigr)\cdot \babs{(\gamma(s)-\target)\times \gamma'(s)}\d
s,
  \label{for:tensor_smooth_new}
  \end{align}
where
\begin{align}
y(\tilde r,s)=\begin{pmatrix}
x_0+r_0(s)\tilde r\cos(\theta(s)) \\
y_0+r_0(s)\tilde r\sin(\theta(s))
\end{pmatrix}.\label{for:srmap_new}
\end{align}
Clearly, the function $f\bigl(y(\tilde r,s)\bigr)$ is
smoother than $f$, since the mapping from radius-arc length coordinates to
Cartesian coordinates is smooth (its inverse is not).
\end{theorem}

\begin{proof}
By expressing the double integral in polar coordinates with the radius
normalized to one, which is legitimate because of the assumption that $\theta(s)$
is strictly monotone, we get
  \begin{align}
&\iint_{\Delta}\log(\norm{\target-y})f(y)\d
A_y\notag\\
&=\int_{\theta(0)}^{\theta(L)}\int_0^1 \log(\norm{\target-y(\tilde
r,s(\theta))})\cdot f\bigl(y(\tilde r,s(\theta))\bigr) \cdot r_0(s(\theta))^2
\cdot \tilde r \d \tilde r \d \theta\notag\\
&= \int_{\theta(0)}^{\theta(L)}\int_0^1 \log\bigl(r_0(s(\theta))\tilde r\bigr)\cdot
f\bigl(y(\tilde r,s)\bigr) \cdot r_0(s(\theta))^2 \cdot \tilde r \d \tilde r
\d \theta.\label{for:polar_tensor_new}
  \end{align}

Then, by a change of variables $\theta=\theta(s)$ and Corollary
\ref{cor:thetap}, (\ref{for:polar_tensor_new}) becomes 
  \begin{align}
&\hspace*{-4em}
\iint_{\Delta}\log(\norm{\target-y})f(y)\d A_y =
\int_{0}^{L}\int_0^1
\log(r_0(s)\tilde r)\cdot f\bigl(y(\tilde r,s)\bigr)\cdot\tilde
r_0(s)^2\cdot\tilde r \cdot \theta'(s)\d \tilde r  \d s\notag\\
&=\int_{0}^{L}\int_0^1
\log(r_0(s)\tilde r)\cdot f\bigl(y(\tilde r,s)\bigr)\cdot \tilde r
\cdot \babs{(\gamma(s)-\target)\times \gamma'(s)}\d \tilde r  \d s\notag\\
&=\int_{0}^{L}\Bigl(\int_0^1 \log(r_0(s)\tilde r)\cdot f\bigl(y(\tilde r,s)\bigr)
\cdot \tilde r \d \tilde r \Bigr)\cdot \babs{(\gamma(s)-\target)\times \gamma'(s)}\d
s.
  \end{align}

\end{proof}

\begin{figure}[h]
    \centering
    \includegraphics[width=0.99\textwidth]{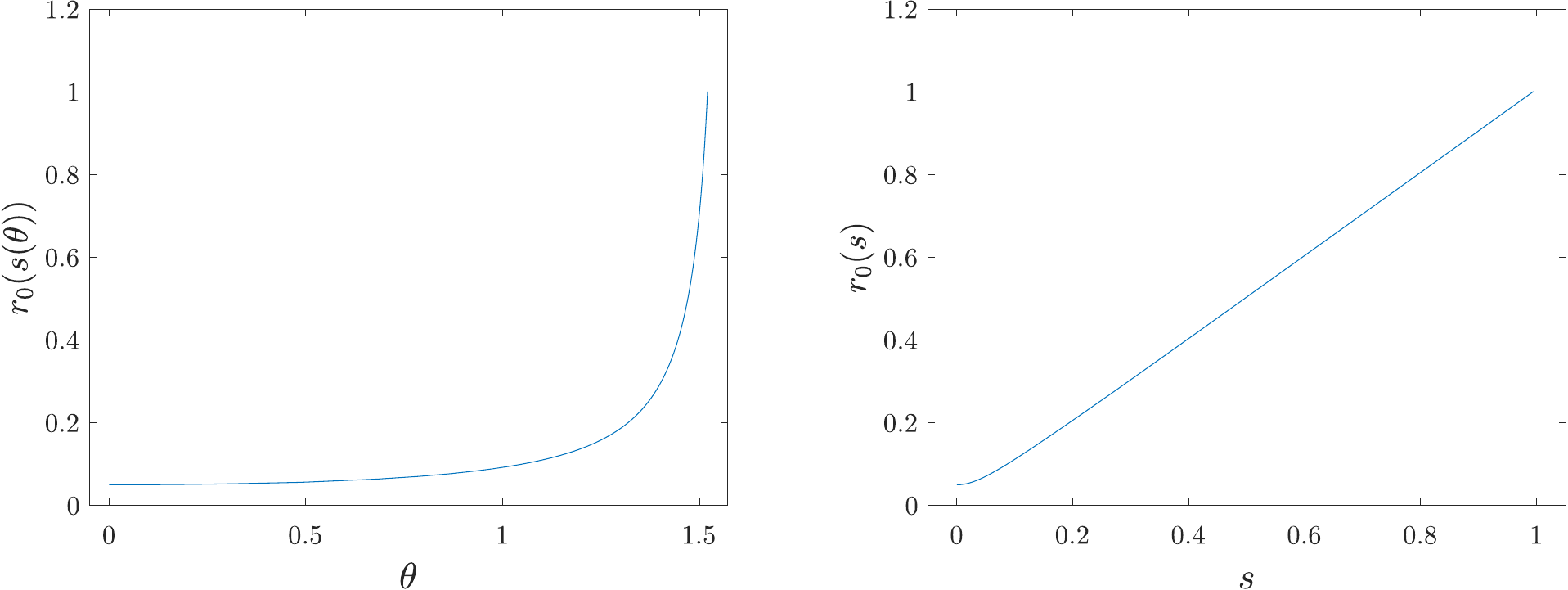}
  \caption{
      {\bf The mapping $r_0(s)$ is better-conditioned than
      $r_0(s(\theta))$}.  In this example, the curved side is the line
      segment connecting $(1,0)$ and $(1,1)$, and the origin of the polar
      coordinate system is $v=(0,0)$.}
   \label{fig:ill_map}

\end{figure}

\begin{observation}
  \label{obs:condition_number_new}
It is also possible to express the double integral over $\Delta$ in polar
coordinates (see formula (\ref{for:polar_tensor_new})). However, when $\Delta$
is stretched, the mapping from the polar angle to the radius (i.e.,
$r_0(s(\theta))$) is ill-conditioned, while the mapping from the arc length
parameter to the radius (i.e., $r_0(s)$) is always well-conditioned (see
Figure \ref{fig:ill_map}).
\end{observation}

\begin{observation}
  \label{obs:smooth_jacobian_new}
When the curve $\gamma$ is a line segment, it is not hard to
show that the function $\babs{(\gamma(s)-\target)\times \gamma'(s)}$ is a
constant, and is equal to the distance from $\target$ to $\gamma$. 
Furthermore, this function is also a constant when $\gamma$ is an arc of a
circle centered at $\target$. In general, this function turns out to be
smooth except when the curve $\gamma$ is highly convex, regardless of the
aspect ratio of the curved element associated with $\gamma$.
It follows that the Jacobian term $\tilde r\cdot
\babs{(\gamma(s)-\target)\times \gamma'(s)}$ in (\ref{for:tensor_smooth_new}) is
also a smooth function under the same conditions. 
Therefore, given a smooth function over a
curved element $\tDelta$ whose curved side is well-resolved and not
excessively convex, the corresponding function in radius-arc length
coordinates is also smooth, regardless of the aspect ratio of $\tDelta$.
Such a property is particularly useful for numerical integration.
\end{observation}

\begin{remark}
To fulfill the assumption in Theorem \ref{thm:tensor_smooth_new}, it is
necessary to have the polar angle of the curved side of each curved element
be strictly monotone. In other words, any ray with the vertex opposite to
the curved side as an endpoint cannot intersect the curved side more
than once. In practice, this assumption can be easily
fulfilled by a slight refinement of the mesh near the domain boundary. We note
that the mesh refinement is easy to achieve by the Distmesh algorithm (see
Observations \ref{obs:adap_mesh}, \ref{obs:curva_fh}). 
\end{remark}

By Theorem \ref{thm:tensor_smooth_new}, (\ref{for:volpot_self}) becomes
  \begin{align}
&\iint_{\Delta}\log(\norm{\target-y})f(y)\d A_y\notag\\
&=\int_{0}^{L}\Bigl(\int_0^1 \bigl(\log(r_0(s))+\log(\tilde
r)\bigl)\cdot f\bigl(y(\tilde r,s)\bigr) \tilde r \d \tilde r \Bigr)
\babs{(\gamma(s)-\target)\times \gamma'(s)}\d s\notag\\
&=\int_{0}^{L}\bigl(I_1(s)+I_2(s)\bigr)
\babs{(\gamma(s)-\target)\times \gamma'(s)}\d s,
\label{for:volpot_self_tensor}
  \end{align}
where 
  \begin{align}
\hspace*{-2em}
I_1(s)=\int_0^1 \tilde r  \log\tilde r\cdot f\bigl(y(\tilde r,s)\bigr) \d \tilde r,\quad
I_2(s)=\int_0^1 \tilde r  \log(r_0(s)) f\bigl(y(\tilde r,s)\bigr) \d \tilde r.
  \end{align}

Clearly, the integrand of $I_2(s)$ is smooth, while the
integrand of $I_1(s)$ has an $\tilde r \log\tilde r$ singularity at $\tilde
r=0$, which has to be resolved by our quadrature scheme. Using 
generalized Gaussian quadratures (see \cite{ggq}), for any given
$N\in \N_{\geq 0}$, we can obtain a set of quadrature nodes and weights
  \begin{align}
\{\tilde r_j,\tilde w_j\}_{j=1,2,\dots,N_q}
  \end{align}
that integrates both
  \begin{align}
\bigl\{\tilde r\log\tilde r \cdot P_n(2\tilde r-1)\bigr\}_{n=0,1,\dots,N}
  \end{align}
and
  \begin{align}
\bigl\{\tilde r \cdot P_n(2\tilde r-1)\bigr\}_{n=0,1,\dots,N}
  \end{align}
over the interval $[0,1]$ to machine precision, where $P_n$ represents the
Legendre polynomial of order $n$.  In other words, this quadrature rule
integrates the space of polynomials of order $N$ over $[0,1]$ multiplied by
$\tilde r$ or $\tilde r\log\tilde r$ to machine accuracy. Therefore, the
inner integral of (\ref{for:volpot_self_tensor}) can be approximated by
  \begin{align}
I_1(s)+I_2(s)\approx \sum_{j=1}^{N_q} \tilde w_j \tilde r_j \log (r_0(s)\tilde r_j )
f\bigl(y(\tilde r_j,s)\bigr),
  \end{align}
from which it follows that  
  \begin{align}
\hspace*{-5em}\iint_{\Delta}\log(\norm{\target-y})f(y)\d A_y
\approx \sum_{j=1}^{N_q} \tilde w_j \tilde r_j \int_0^L \log (r_0(s)\tilde r_j )
f\bigl(y(\tilde r_j,s)\bigr)\babs{(\gamma(s)-\target)\times \gamma'(s)}\d s\notag\\
=\sum_{j=1}^{N_q} \tilde w_j \tilde r_j \int_0^L \bigl(\log(r_0(s))+\log(\tilde
r_j)\bigr) f\bigl(y(\tilde r_j,s)\bigr)\babs{(\gamma(s)-\target)\times \gamma'(s)}\d s.
\label{for:volpot_self_single}
  \end{align}
To compute the integrals with respect to the arc length $s$ in
(\ref{for:volpot_self_single}), we first observe that $f\bigl(y(\tilde
r_j,s)\bigr)\cdot\babs{(\gamma(s)-\target)\times \gamma'(s)}$ is a smooth
function of $s$ when $\gamma$ is smooth and not too convex (see
Observation~\ref{obs:smooth_jacobian_new}).  Thus, all that remains is to
examine the singularity of the term $\log (r_0(s))$. In the case where
$\gamma$ denotes a line segment, it is clear
that
  \begin{align}
r_0(s)=\sqrt{(s-\tilde s)^2+\norm{\target-\gamma(\tilde s)}^2}
  \end{align}
for $\tilde s:=\argmin_s \norm{\target-\gamma(s)}$, from which it follows
that the analytic continuation of $r_0(s)$ equals zero at $s=\tilde
s+i\norm{\target-\gamma(\tilde s)}$, where $i$ is the imaginary unit.
Thus, when $\gamma$ is a line segment, $\log(r_0(s))$ has a singularity at
$s=\tilde s+i\norm{\target-\gamma(\tilde s)}$. In the case where $\gamma$ is
an arbitrary curved side that is well-resolved by the mesh (which is
one of our assumptions), $\gamma$ can be seen as a smooth perturbation
of a line segment.  If we define $\tilde s$ by the formula described above
for the line segment, the singularity of $\log(r_0(s))$ is likewise
located near $s=\tilde s+i\norm{\target-\gamma(\tilde s)}$. Below, we introduce
an algorithm for resolving such a singularity.

\begin{enumerate}
\item Subdivide the integration interval into two pieces: $[0,\tilde s]$ and
$[\tilde s, L]$, where $\tilde s=\argmin_s \norm{\target-\gamma(s)}$.
\item Divide $[0,\tilde s]$ into $N+1$ subintervals
  \begin{align}
\{[(1-(1/2)^{N})\tilde s,\tilde s]\}\bigcup
\{[(1-(1/2)^{i-1})\tilde s,(1-(1/2)^{i})\tilde s]\}_{i=1,2,\dots,N},
  \end{align}
where 
  \begin{align}
N=\min \{N: (1/2)^{N}\tilde s <\norm{\target-\gamma(\tilde s)}\}.
  \end{align}
\item Similarly, divide $[\tilde s, L]$ into $M+1$ subintervals, where
  \begin{align}
M=\min \{M: (1/2)^{M}(L-\tilde s) <\norm{\target-\gamma(\tilde
s)}\}.
  \end{align}
\item Use the Gauss-Legendre quadrature rule of order $p$ over each subinterval
to compute~(\ref{for:volpot_self_single}). It is necessary for $p$ to be large 
enough, such that the two integrals
  \begin{align}
\int_0^{\tilde s/2} \bigl(\log(r_0(s))+\log(\tilde
r_j)\bigr) f\bigl(y(\tilde r_j,s)\bigr)\babs{(\gamma(s)-\target)\times \gamma'(s)}\d s
  \end{align}
and 
  \begin{align}
\int_{(L+\tilde s)/2}^{L} \bigl(\log(r_0(s))+\log(\tilde
r_j)\bigr) f\bigl(y(\tilde r_j,s)\bigr)\babs{(\gamma(s)-\target)\times \gamma'(s)}\d s
  \end{align}
can be computed to full accuracy. 
\end{enumerate}

Using the subdivision technique presented above, it is guaranteed that each
panel is separated from the singularity by at least one panel width, such
that it is in the far field of the logarithmic singularity. It follows that
an accurate quadrature result is obtained.

\begin{remark}
\label{rem:self_precomp}
The adaptive subdivision along the arc length coordinate for resolving the
nearly-singular integrand can be expensive, especially when the target is
close to the boundary. One can instead use a large number of precomputed
specialized quadrature rules to resolve the nearly-singular integrand, which
is substantially more efficient than the adaptive integration method
described above (see \cite{bremer1,bremer2}). In our implementation, for simplicity, we
compute the self-interactions using the adaptive integration method. 
\end{remark}

\begin{remark}
The need for subdivisions in the arc length coordinate is not an
artifact of the change of variables, since when the target location $\target$ is
close to the edge of a mesh element, the integrand is in nature
nearly-singular along the arc-length direction in the region between $\target$ and 
that edge.
\end{remark}

\subsection{Interpolation over a mesh element}
  \label{sec:interp}
In this section, we describe an interpolation scheme of the volume potential
(\ref{for:volpot}) over a mesh element. We note that the volume
potential is smooth provided that the boundary of the domain
$\partial\Omega$ and the density $f$ is smooth (see, for example, \cite{ying1}),
from which it follows that the Koornwinder polynomial basis is suitable
for the interpolation of the volume potential in this case.

First of all, given a mesh element $\Delta$ and an invertible and
well-conditioned mapping $\rho$ from the standard simplex $\tri$ to $\Delta$
(see Section \ref{sec:int_smooth} for the construction of $\rho$), 
the Koornwinder polynomial expansion coefficients of the function $u\circ
\rho:\tri\to \R$ can be computed by applying the interpolation matrix
to the values of $u$ at the Vioreanu-Rokhlin nodes over $\Delta$ (see
Section \ref{sec:vioreanu}), and we denote the interpolant by $I_{u\circ
\rho}$. As $u\circ \rho$ is a smooth function and $\rho$ is easily
invertible (the use of Newton's method is necessary when $\Delta$ is a
curved element), we can compute the potential to high accuracy at any point
on $\Delta$ by evaluating $(I_{u\circ \rho})\circ \rho^{-1}$.

\begin{remark}
  \label{rem:bndry_interp}
In some applications, it is useful to compute the volume potentials at the
boundary of the domain, i.e., the value of $u(v)$ for some $v=\gamma(s)$,
where $\gamma$ is the arc length parameterization of the curved side of the
curved element that $v$ lies in.  In this case, the use of Newton's method
is unnecessary, since it is easy to show that $\rho^{-1}(v)=(1-s/L,0)$, where
$\rho$ is the blending function mapping (\ref{for:blend}), and $L$ is the
total arc length of $\gamma$.
\end{remark}

\begin{observation}
As is noted in Observation \ref{obs:nonsm_bjac}, when the curved
side of a curved element is not well-resolved by the mesh, the Jacobian 
of $\rho$ could be nonsmooth. Thus, it is important for the curved side
to be well-resolved by the mesh, or, if it is not well-resolved, to use a
higher-order interpolation scheme.
\end{observation}

\subsection{Coupling quadratures to the FMM}
  \label{sec:fmm}

With the numerical apparatus developed in the previous sections, given 
geometry with a triangle mesh, and a source function $f:\R^2\to \R$, we can
evaluate the integral (\ref{for:volpot}) at any given target location in
$\O(N)$ operations, where $N$ is the total number of the quadrature nodes,
by summing up all of the potentials generated over the mesh elements at the
target location, using the algorithms introduced in Sections \ref{sec:far_int},
\ref{sec:near_int}, \ref{sec:self_int}.  It is not hard to see that the cost
is dominated by the far field interactions.

It is often desirable to evaluate the potential at multiple
targets, which, if computed naively, leads to $\O(MN)$ total operations,
where $M$ is the number of target locations. Such an expensive cost is
usually prohibitive in practice.  In this section, we reduce the time
complexity to $\O(M+N)$ by coupling the FMM to our volume potential evaluation
algorithm.

The key observation of the algorithm is that, if we treat all interactions
as far field interactions, the volume potential evaluation is equivalent to
a standard electrostatic interaction problem, where the charge locations are
the quadrature nodes over the triangles, and the charge densities are given
by the density function values at the nodes multiplied by the
corresponding quadrature weights. Therefore, with the well-known fast
multipole method \cite{fmm}, such computations can be done with linear time
complexity.  However, since treating all quadrature nodes as point sources
does not account for the near and self-interactions, the computed potential
is inaccurate. Therefore, it is necessary to make a correction to the
computed potential. Below, we outline the so-called ``subtract-and-add''
method (see \cite{leslie}) for the potential corrections.

Suppose that $x$ is a target location, and $u(x)$ is the potential at $x$
computed using the FMM. Let $\st(x)$ denote the element containing $x$,
$\nt(x)$ denote the near field of $x$, as defined in Definition
\ref{def:nearby_mesh_type}.
\begin{enumerate}
\item Find the mesh element $\st(x)$ that $x$ lies within, and
the set of nearby elements $\nt(x)$ using the query algorithm
described in Appendix \ref{sec:query}. 
\item Subtract from $u(x)$ the contribution made in the FMM calculation from
the quadrature nodes over $\nt(x)$ and $\st(x)$. Then, add to $u(x)$ the
near and self-interaction potential over $\nt(x)$ and
$\st(x)$, respectively, using the algorithms described in
Sections \ref{sec:near_int} and \ref{sec:self_int}.
\end{enumerate}

Since the number of mesh elements in the near field of a given target is
$\O(1)$, we have that the number of operations required to correct the
potential at each target is $\O(1)$. Therefore, the total time complexity of
the algorithm is still $\O(M+N)$.

\begin{remark}
Clearly, the subtraction of spurious contributions from near fields can be
avoided by disabling the computation of neighbor interactions in the
FMM. However, it is observed in \cite{leslie} that the use of the
``subtract-and-add'' method described in this Section leads to faster
implementation, largely because of a better cache utilization.
We note that catastrophic cancellation could happen during the
``subtract-and-add'' stage. In our case, the singularity of the logarithmic
kernel is so mild that this is generally not a concern. When a more singular
kernel is in use, one should omit the computation of neighbor interactions
in the FMM, and handle these interactions using the algorithms of Sections
\ref{sec:near_int}, \ref{sec:self_int}.
\end{remark}

\section{Accelerating potential evaluation over unstructured meshes}
  \label{sec:acceleration}

The algorithm described thus far, while differing in some details, involves
fairly standard ideas that have been developed for the evaluation of the
surface potential \cite{bremer1,leslie}.  In this section, we describe three
general observations that lead to accelerated volume potential calculations
over unstructured meshes.
The first concerns the geometry of the near field. We observe that the near
field is typically approximated by a ball or a triangle (see, for example,
\cite{bremer1,leslie,anderson}), while in reality, the geometry of the near
field can be characterized precisely by the union of several Bernstein
ellipses. Moreover, these ellipses are very flat when the order of the far
field quadrature rule is high, or when the error tolerance is large, from
which it follows that the conventional approximations of the near field
become inaccurate. 
When the near field is computed accurately, the cost of the most expensive
component in the volume potential evaluation algorithm, namely, the near
interaction potential correction, is reduced substantially. The second
observation is that, knowing the near field geometry, one can
efficiently offload the near field interaction corrections onto the
FMM-based far field interaction computation by increasing the order of the
far field quadrature rule. Such a trade leverages the computational efficiency of
highly optimized parallel FMM libraries, and reduces the cost of the much
more expensive and unstructured near field interaction computation. The
third observation is that the most commonly used arrangement, in which the
interpolation nodes are placed on the same mesh over which the potential is
integrated over, leads to an artificially large potential correction cost.
If the quadrature mesh and the interpolation mesh are instead staggered to
one another, the number of interpolation nodes at which the near field and
self-interactions are expensive to evaluate is reduced dramatically.

We present these techniques in the context of completely unstructured
meshes, and note that the ideas are equally applicable in the context of
hybrid or structured meshes.

\subsection{Precise near field geometry analysis}
  \label{sec:nf}

In this section, we present the precise characterization of the geometry of 
the near field, the use of which eliminates the unnecessary near field
interaction computations, and provides an error control functionality to the
far and near field interaction computations.

By Definition \ref{def:far_near_field}, given a triangle integration domain
$\bDelta$, a set of possible densities $S$, a far field quadrature rule
$\{y_i,w_i\}_{i=1,2,\dots,M}$ over $\bDelta$ of order $N$ with length $M$,
and an error tolerance $\varepsilon$, the far field of $\bDelta$ is the set
of points
  \begin{align}
\hspace*{-5.5em}\ft_{\bDelta}=\{x\notin \bDelta: \bbabs{\sum_{i=1}^M
w_i\log\norm{x-y_i}f(y_i) - \iint_{\bDelta} \log\norm{x-y}f(y)\d A_y} <
\varepsilon, f\in S\},
  \label{for:precise_far}
  \end{align}
and the near field of $\bDelta$, which we denote by $\nt_{\bDelta}$, equals
the complement of $\ft_{\bDelta}$ in the set $\R^2\setminus \bDelta$.

For simplicity, we assume that $\bDelta$ is the standard simplex $\tri$.
Given an arbitrary point $x_0=(a,b)\in \ft_{\tri}$, we rewrite the integrand
as a Koornwinder polynomial expansion
  \begin{align}
R(y):=\log\norm{x_0-y}f(y)=\sum_{n=0}^\infty\sum_{m=0}^n a_{mn}K_{mn}(y),
  \label{for:defR}
  \end{align}
where $a_{mn}$ is the inner product between $f$ and $K_{mn}$ over $\tri$.
Furthermore, by definition (\ref{for:precise_far}), we have that
  \begin{align}
&\bbabs{\sum_{i=1}^M w_i R(y_i) - \iint_{\tri}
 R(y)\d A_y}<\varepsilon.
  \label{for:precise_far_spec}
  \end{align}
By assumption, the far field quadrature rule $\{y_i,w_i\}_{i=1,2,\dots,M}$ integrates
Koornwinder polynomials of order up to $N$ exactly, from which it follows that
  \begin{align}
\sum_{i=1}^M w_i\sum_{n=0}^N\sum_{m=0}^{N-n} a_{mn}K_{mn}(y_i) = \iint_{\tri}
\sum_{n=0}^N\sum_{m=0}^{N-n} a_{mn}K_{mn}(y)\d A_y. 
  \label{for:precise_far_simp}
  \end{align}
Thus, by combining (\ref{for:precise_far_simp}) and the orthogonality of
the Koornwinder polynomials (see (\ref{for:koorn_ortho})), the inequality
(\ref{for:precise_far_spec}) becomes
  \begin{align}
\bbabs{\sum_{i=1}^M w_i\cdot R_N(y_i)}<\varepsilon,
  \label{for:residue_eps}
  \end{align}
where
  \begin{align}
R_N(y):=&\, \sum_{n=0}^\infty\sum_{m=0}^n a_{mn}K_{mn}(y)-
\sum_{n=0}^N\sum_{m=0}^{N-n}
a_{mn}K_{mn}(y)\notag\\
=&\, \sum_{n=\ceil{(N+1)/2}}^\infty\sum_{m=N-n+1}^{n}
a_{mn}K_{mn}(y)
  \label{for:remainder}
  \end{align}
denotes the remainder, i.e., the sum of the expansion terms with order
larger than $N$. 
Since the remainder $R_N(y)$ is a general function, it is generally true that
(\ref{for:residue_eps}) holds if and only if 
  \begin{align}
\babs{w_i \cdot R_N(y_i)}<\varepsilon,
  \label{for:single_residue_eps}
  \end{align}
for all $i$. Specifically, the inequality
(\ref{for:single_residue_eps}) holds only if
  \begin{align}
\babs{R_N(y_j)}<\frac{\varepsilon}{w},
  \label{for:single_residue_eps_2}
  \end{align}
where $\{y_j\}$ denotes the $\sim n/2$ nodes that are closest to $[0,1]\times
\{0\}$, and $w$ denotes the largest quadrature weight corresponding to
these nodes (we note that there are usually around $n/2$ Xiao-Gimbutas nodes
along $[0,1]\times \{0\}$). By continuity of the remainder $R_N$,
(\ref{for:single_residue_eps_2}) implies that 
  \begin{align}
\bnorm{R_N|_{[0,1]\times \{0\}}}_{L^\infty}<\varepsilon\cdot C_N /w,
  \label{for:trace_eps}
  \end{align}
where $R_N|_{\partial \tri}$ is the trace of the remainder $R_N$ 
on the edges of the standard simplex $\tri$, and $C_N$ is a constant
that converges down to $1$ as $N$ increases, by continuity of $R_N$.
We postpone the discussion on the actual value of $C_N$ to
Observation \ref{obs:empirical_rho}.

Without loss of generality, we only consider the trace of $R_N$
on the edge connecting the two points $(0,0)$ and $(1,0)$, i.e.,
$R_N\big|_{[0,1]\times \{0\}}$. Furthermore, for ease of presentation, 
we map its domain to $[-1,1]$ by defining $H_N:[-1,1]\to\R$, where
  \begin{align}
H_N(x):=R_N\big|_{[0,1]\times\{0\}}\Bigl(\frac{x+1}{2}\Bigr).
  \end{align}
Similarly, we can define the trace of the integrand $R$ (see
(\ref{for:defR})) on the same edge (after the same
linear transformation of the domain) by $H:[-1,1]\to\R$, where
  \begin{align}
H(x):=R\big|_{[0,1]\times
\{0\}}\Bigl(\frac{x+1}{2}\Bigr)=\log\Babs{a+ib-\frac{x+1}{2}}\cdot 
f\Bigl(\frac{x+1}{2}\Bigr),
  \end{align}
and $i$ is the imaginary unit (recall that $a,b$ denote the
$x,y$-coordinates, respectively, of the given point $x_0$).
Clearly, $H-H_N$ is the projection of $H$ onto the space of one-dimensional
polynomials of order $N$.

With $H$ and $H_N$ thus defined, (\ref{for:trace_eps}) can be rewritten as 
  \begin{align}
\bnorm{H_N}_{L^\infty[-1,1]}<\varepsilon\cdot C_N/w.
  \label{for:Hn_eps}
  \end{align}
We now define a Bernstein ellipse $E_{\rho_0}$ with parameter
  \begin{align}
\rho_0=\bigl(\frac{w}{\varepsilon\cdot C_N }\bigr)^{1/N}.
  \label{for:rho0}
  \end{align}
It is clear that $E_{\rho_0}$ is well-defined for sufficiently small
$\varepsilon$ (so that $\rho_0>1$).
Then, we have that $H-H_N$ is a one-dimensional polynomial of
order $N$ approximating $H$, satisfying
  \begin{align}
\norm{H-(H-H_N)}_{L^\infty[-1,1]}=\norm{H_N}_{L^\infty[-1,1]}<\varepsilon\cdot 
 C_N/w=\rho_0^{-N},
  \end{align}
by (\ref{for:Hn_eps}) and (\ref{for:rho0}). It follows that, by Theorem
\ref{thm:bern},
  \begin{align}
H(z)=\log\Babs{a+ib-\frac{z+1}{2}}\cdot f\Bigl(\frac{z+1}{2}\Bigr)
  \end{align}
can be analytically continued to an analytic function in the open
Bernstein ellipse $E_{\rho_0}$. Therefore, we have that the complex 
number $2(a+ib)-1$ must be outside of $E_{\rho_0}$. Equivalently, in the
Cartesian plane, the point $x_0$ has to be outside of
  \begin{align}
E_{\rho_0}^{(1)}:=\Bigl\{\Bigl(\frac{x+1}{2},\frac{y}{2}\Bigr):x+iy\in
E_{\rho_0}\Bigr\}.
  \label{for:Erho0}
  \end{align}
Similarly, one can apply the same argument to the remaining two edges of $\tri$,
and get two more ellipses, which we denote by $E_{\rho_0}^{(2)}$ and
$E_{\rho_0}^{(3)}$. It follows that it is a necessary condition that
the near field $\nt_{\tri}$ contains the union of the three
ellipses in $\R^2\setminus \tri$, i.e., 
  \begin{align}
E_{\tri}\subset \nt_{\tri},
  \label{for:far_nece}
  \end{align}
where
  \begin{align}
E_{\tri}:=\bigl(\cup_{i=1}^{3}E_{\rho_0}^{(i)}\bigr)\setminus\tri.
\label{for:threeE}
  \end{align}
In other words, $E_{\tri}$ is a lower bound of the near field $\nt_{\tri}$.
One could, in fact, apply the same argument to get an ellipse for every line
segment inside $\tri$, however, it is clear that their union in
$\R^2\setminus\tri$ equals $E_{\tri}$. Thus, $E_{\tri}$ is the optimal
lower bound that we can obtain using our argument.
Furthermore, it turns out that, with some additional mild assumptions, one
can show that the lower bound $E_{\tri}$ is tight, and we sketch the proof
below.

Firstly, we define a union of three Bernstein ellipses that are slightly
larger than $E_{\tri}$:
  \begin{align}
E_{\tri}':=\bigl(\cup_{i=1}^{3}E_{C_N^{1/N}\cdot \rho_0}^{(i)}\bigr)\setminus\tri.
\label{for:threeEprime}
  \end{align}
We note that the factor $C_N^{1/N}$ converges down to $1$ as $N$ increases,
and that empirically, $C_N$ is small when $N$ is small (see Table \ref{tab:CN_values}).

Given $x_0\in\R^2\setminus (\tri\cup E_{\tri}')$, suppose that the
restriction of the integrand $R$ to any line segment inside $\tri$ can be
analytically continued inside $\tri\cup E_{\tri}'$.  Suppose further that the
analytic continuation of each such restriction is bounded inside $\tri\cup
E_{\tri}'$ by $1/4$. Also note that in practice, the parameter $\rho_0$ is
generally larger than $1.5$ (see formula (\ref{for:rho0}), Observation
\ref{obs:empirical_rho}).  We want to show that $x_0\in\ft_{\tri}$.

Recall that $x_0\in\ft_{\tri}$ if 
  \begin{align}
\babs{R_N(y_i)}<\frac{\varepsilon}{w_i},
  \label{for:res_eps_3}
  \end{align}
for all $i$ (see formulas (\ref{for:remainder}),
(\ref{for:single_residue_eps})). We show that (\ref{for:res_eps_3}) holds
for all the quadrature nodes $\{y_i\}$ by considering two cases, one 
for the $\sim 3\cdot n/2$ nodes that are closest to $\partial \tri$, and the
other for the rest of the nodes, which are away from $\partial \tri$.

\paragraph{Proof for $\{y_i\}$ that are close to $\partial\tri$:}
By Corollary \ref{cor:anal_conv}, we have that
  \begin{align}
\hspace*{-5.5em}\norm{R_N}_{L^\infty([0,1]\times\{0\})}=\norm{H_N}_{L^\infty[-1,1]}\approx
\norm{H-f_N}_{L^\infty[-1,1]} \leq (C_N^{1/N}\cdot \rho_0)^{-N}=\varepsilon /w,
  \end{align}
where $f_N$ denotes the Chebyshev projection of
$H$ of order $N$. By a similar argument, one can show that
  \begin{align}
\norm{R_N}_{L^\infty([0,1]\times\{y\})}\leq \varepsilon /w\leq\varepsilon/w_i,
  \end{align}
for all $y$ close to zero, and for all
quadrature weights $w_i$ corresponding to the nodes $y_i$ that are close to
$[0,1]\times\{0\}$.
Therefore, we have that (\ref{for:res_eps_3}) holds for all the $\sim n/2$
nodes $y_i$ that are closest to the line segment connecting $(0,0)$ and
$(1,0)$. Similarly, one can show that (\ref{for:res_eps_3}) holds for all
the $\sim 3\cdot n/2$ nodes that are closest to the three sides of $\tri$.

\paragraph{Proof for $\{y_i\}$ that are away from $\partial\tri$:}
We observe that, in practice (i.e., when the order $N$ of the quadrature
rule is no larger than $50$), all the quadrature weights $\{w_i\}$ are
bounded by $10\cdot w$.
Therefore, to prove (\ref{for:res_eps_3}), it is sufficient to show that 
  \begin{align}
\babs{R_N(y)}<\frac{\varepsilon}{10\cdot w},
  \label{for:res_eps_4}
  \end{align}
for $y\in \tri$ that are away from $\partial \tri$. Numerically, given an
arbitrary line segment $L$ inside $\tri$, provided that the mapping from
$\{K_{mn}|_L\}_{0\leq m+n\leq N}$ to $\{T_n\}_{n=1,2,\dots,N}$ is stable
(where $T_n$ denotes the Chebyshev polynomial of order $n$),   
  \begin{align}
\norm{R_N|_L}<\frac{\varepsilon}{10\cdot w}
  \label{for:res_eps_5}
  \end{align}
holds if and only if
  \begin{align}
\norm{R|_L-f_n}<\frac{\varepsilon}{10\cdot w}
  \label{for:res_eps_6}
  \end{align}
holds, where $f_n$ is the $n$th order Chebyshev projection of $R|_L$.
Clearly, for each $y\in\tri$, one can pick a
sufficiently long line segment $L$ that contains $y$, such
that the mapping is stable. Then, by Corollary \ref{cor:anal_conv}, the
inequality (\ref{for:res_eps_6}) (equivalently, (\ref{for:res_eps_5})) holds
if $x_0$ is outside a Bernstein ellipse transformed so that the interval
$[-1,1]$ corresponds to $L$, with parameter
$(10\cdot w/\varepsilon)^{1/N}$. It is easy to see that such an ellipse is
inside $E_{\tri}\cup \tri$. Thus, by definition of $x_0$, the inequality
(\ref{for:res_eps_4}) holds for all $y\in \tri$ that are away from
$\partial \tri$.

\paragraph{}
Therefore, (\ref{for:res_eps_3}) holds for all $i$, from which it
follows that $x_0\in\ft_{\tri}$, and $E_{\tri}$ is a tight lower bound of
the near field when $C_N^{1/N}$ is close to $1$.

\begin{observation}
Empirically, we find that, when $f\in S$ is sufficiently smooth over the
integration domain (a triangle), and when $C_N$ is chosen according to Table
\ref{tab:CN_values}, the set $E_{\tri}$ precisely characterizes the near
field of the standard simplex $\tri$ (see Figures \ref{fig:nf1},
\ref{fig:nf0}, \ref{fig:nf2}).
  \label{obs:empirical_rho}
\end{observation}

\begin{table}[h!!]
    \begin{center}
    \begin{tabular}{c|cccccc}
    $N$ & 50 & 40 & 30& 25&20 &12 \\
    \hline
      $C_N$ & 1 & 2 & 2.5 & 3 & 6.8 & 12 
    \end{tabular}
    \caption{
    {\bf The values of $C_N$ that allow $E_{\tri}$ to 
    characterize the near field of $\tri$ precisely.} We note that this table is
    obtained empirically, and the values work for arbitrary triangles.
    Furthermore, the table agrees with the claim that $C_N$ converges down
    to $1$ as $N$ increases. 
    }
    \label{tab:CN_values}
    \end{center}
\end{table}

To generalize this argument to arbitrary triangles, one
only needs to adjust the constant $w$ (see (\ref{for:single_residue_eps_2}))
based on the magnitude of the quadrature weights (equivalently, the size of
the triangle). This, however, does not add difficulty to the implementation,
since one only needs to compute $w$ for the standard simplex once, and then
scale it according to the ratio between the area of the triangle and
the area of the standard simplex. We note that the naive estimation
of the near field neglects this nonlinear relation between the size of the
near field and the size of the mesh element, which causes unnecessary near field
interaction computations, especially when the mesh element is small.
Apart from this, it is important to note
that, given a stretched triangle, $E_{\tri}$ may not be the optimal lower
bound that we can obtain using our argument. For example, as is shown in the
left part of Figure \ref{fig:nf1_d}, the near field corresponding to the
error tolerance $10^{-14}$ is not perfectly captured by $E_{\tri}$. Such an
issue can be fixed by adding the ellipse, obtained from applying our
argument to the line segment connecting $(0,0.5)$ and $(4,0.5)$, to the
union of ellipses. In practice, stretched triangles rarely appear if a
decent meshing algorithm is used, and the presence of a few stretched
triangles barely affects the overall accuracy of the evaluation.

In the situation where the domain is a curved element,
if the mapping from the reference domain (i.e., a standard simplex) to the
curved element is also valid and sufficiently smooth for points near the
simplex, one can apply our argument to the reference domain, and compute the
inverse mapping of any given point outside the curved element (by Newton's
method) to check whether the point is inside the near field or not;
Alternatively, one could linearize the curved boundary of the curved element with
a polyline, and apply our argument to every linearized boundary segment
(adjusting $\varepsilon$ to account for the numerically smaller polynomial
orders on the traces). We note that, in practice, given a curved element,
very few discretization nodes that are close to the curved side belong to
the near field of the curved element, and thus, it is often convenient to to
treat the curved side as a straight line segment when one applies the near
field geometry analysis.

Besides allowing for the precise identification of all of the necessary
near interaction potential corrections, we note that our near field
geometry analysis is also helpful in the near field interaction 
computation itself. Recall that, in Section \ref{sec:near_int}, we describe
an adaptive algorithm for resolving the nearly-singular integrand, which
recursively subdivides the integration domain, such that the integrand is
smooth on each subdomain. When our near field geometry
analysis is applied to the subdomains, the algorithm is able to decide the 
number of required subdivisions precisely, avoiding the possibility both of
oversampling and undersampling.

\begin{observation}
The shape of a near field is similar to a circle, when both the error
tolerance and the order of the quadrature rule are low (see Figures
\ref{fig:nf1_a} and \ref{fig:nf2_a}). This, however, does not imply that our
estimation is useless in such a setting. 
Firstly, without rigorous analysis, one often needs to overestimate the size
of the near field to improve the robustness of the algorithm. Secondly,
in the adaptive subdivision-based near field interaction computation,
as in the case of arbitrary triangles, one has to scale the
size of $w$ as the area of the triangle becomes smaller during the subdivisions,
which is equivalent to increasing the error tolerance $\varepsilon$ by
(\ref{for:rho0}). One can observe from Figure \ref{fig:nf2_a} that the naive
estimation of the near field of the sub-triangle becomes inefficient, and
leads to many unnecessary subdivisions.
\end{observation}

\begin{remark}
By formulas (\ref{for:rho0}) and (\ref{for:threeE}), the volume of the near
field goes to zero as the order of the far field quadrature rule goes to
infinity. The near field estimated by a ball, on the other hand,
always results in a non-negligible volume. Moreover, we note that,
the more distorted a mesh element is, the poorer the naive estimation
of its near field becomes.
  \label{rem:near_zero}
\end{remark}

\begin{remark}
The shape information of the ellipses can be efficiently
precomputed for all the mesh elements. Thus, the cost of checking whether a
target is inside the approximated near field or not using ellipses is
negligible.
\end{remark}

We report the true near field and our estimated near field, for different
triangles with different densities and different quadrature orders,
in Section \ref{sec:experiment_bern}. We also report the performance of the
volume potential evaluation algorithm, with and without precise near field
geometry analysis, in the same section.

\subsection{Offloading the near field interaction computation onto the
FMM-based far field interaction computation}
\label{sec:offload}

Since the volume of the near field goes to zero as the order of the far
field quadrature rule increases (see Remark \ref{rem:near_zero}), the number
of near field interaction potential corrections can be reduced arbitrarily,
in exchange for a more expensive far field interaction computation.
It has been long realized that one needs to adjust the order of the far 
field quadrature rule, such that the costs of the far field interactions
and the near field interactions are balanced (see, for example,
\cite{anderson,leslie}). However, this idea is presented as a heuristic in
the literature, and the adjustments are done empirically. In fact,
without characterizing the near field geometry precisely,
such an idea cannot be efficiently carried out. This is because, when the
order of the far field quadrature rule is high, the standard approximation
of the near field by a ball or a triangle becomes inaccurate (see Section
\ref{sec:nf} and Figure \ref{fig:nf1}). It follows that the high-order far
field quadrature rule is underutilized; moreover, one has to
overestimate the size of the near field to improve the robustness of the
algorithm, in the absence of the precise near field geometry analysis.
Therefore, the use of the precise near field geometry analysis is critical
for efficiently offloading the near field interaction computation onto the
FMM-based far field interaction computation.

To quantify the trade-off, we analyze the rate at
which the Bernstein ellipse $E_{\rho_0}$ shrinks. 
It is easy to show that the area of a Bernstein ellipse with
parameter $\rho$ is asymptotically proportional to 
  \begin{align}
\log\rho_0=\frac{1}{N}\log\bigl(\frac{w}{\varepsilon\cdot C_N}\bigr),
  \end{align}
(see, for example, \cite{garritano}), from which it follows that the cost of
the near field interaction potential corrections is proportional to
  \begin{align}
\frac{3}{N}\log\bigl(\frac{w}{\varepsilon\cdot C_N}\bigr).
  \end{align}
We also have that the cost of the FMM-based far field interaction
computation is proportional to
  \begin{align}
N_{tgt}+N_{src}=N_{tgt}+\O(N^2),
  \end{align}
where $N_{tgt},N_{src}$ denote the number of the target locations and the number
of quadrature nodes in total, respectively. Despite that the FMM cost
can potentially increase at a faster rate than the near field interaction
potential correction cost decreases, such a trade leverages the computational
efficiency of highly optimized parallel FMM libraries (see, for example,
\cite{par1,par2}), and reduces the cost of the much more expensive and
unstructured near field interaction computation. In practice, we find that
the FMM cost increases much more slowly, since the number of the
interpolation nodes tends to be large compared to the number of the
quadrature nodes of the same order, from which it follows that the FMM cost
is dominated by the large number of target points, unless the order of the
far field quadrature rule is extremely high (see Figure \ref{fig:offload}).  In
addition, although the near field interaction (and self-interaction)
potential corrections are embarrassingly parallelizable, and their costs can
be potentially made very small with the use of many cores, it still requires
engineering efforts to attain the optimal parallel efficiency. On the other
hand, many efforts have been made in the design and implementation of
parallel FMM libraries, and thus, it is preferable to offload the near field
interaction computations onto the far field interaction computations.

\begin{remark}
We use high order far field
quadrature rules in this paper to resolve the Green's function, rather
than the density function. It follows that the density function can be
oversampled during the integration process, which is undesirable when its
evaluation is expensive. In this situation, it is recommended to construct a
lower-order interpolant of the density function.
\end{remark}

We demonstrate the effectiveness of the offloading technique in Section
\ref{sec:experiment_trade}.

\subsection{Fast interpolation of the volume potential with a staggered
mesh}
\label{sec:stagger}

In Section \ref{sec:interp}, we described an interpolation scheme
for the volume potential over a mesh element. In practice, it is 
often desirable to have the volume potential interpolated over all
the mesh elements in the domain $\Omega$, such that the evaluation
of the volume potential at any point in the domain is both instantaneous
and accurate. A common strategy is to let a single mesh serve both
as the quadrature mesh and the interpolation mesh, i.e., the interpolants
are constructed by evaluating the volume potential at the Vioreanu-Rokhlin
nodes over all mesh elements, and the integration domain $\Omega$ is
discretized into the same set of the mesh elements. 
In this section, we first show that such an approach does not lead to 
optimal computational efficiency. Then, we propose a simple modification that
significantly improves the efficiency.

First of all, we note that the time cost of the potential correction 
is strongly correlated with the location of the target: if the
target is extremely close to some edge in the mesh, extensive subdivisions
are required to resolve the near-singularity in the near and
self-interactions (see Sections \ref{sec:near_int}, \ref{sec:self_int} for
details); if the target is away from all of the edges, e.g., in the center
of a mesh element, very few subdivisions are required and the correction can
be made rapidly. Therefore, when a single mesh is both used for quadrature
and interpolation, the potential corrections become very expensive,
since the interpolation nodes tend to cluster around the edges and corners
of the mesh elements (see Figure \ref{fig:cluster}). Furthermore, the
precise identification of the near field and the offloading technique
(described in Sections \ref{sec:nf}, \ref{sec:offload}) become less helpful
if the majority of the nodes nearby a mesh element are indeed inside its near
field.

\begin{figure}[h]
    \centering
    \includegraphics[width=0.9\textwidth]{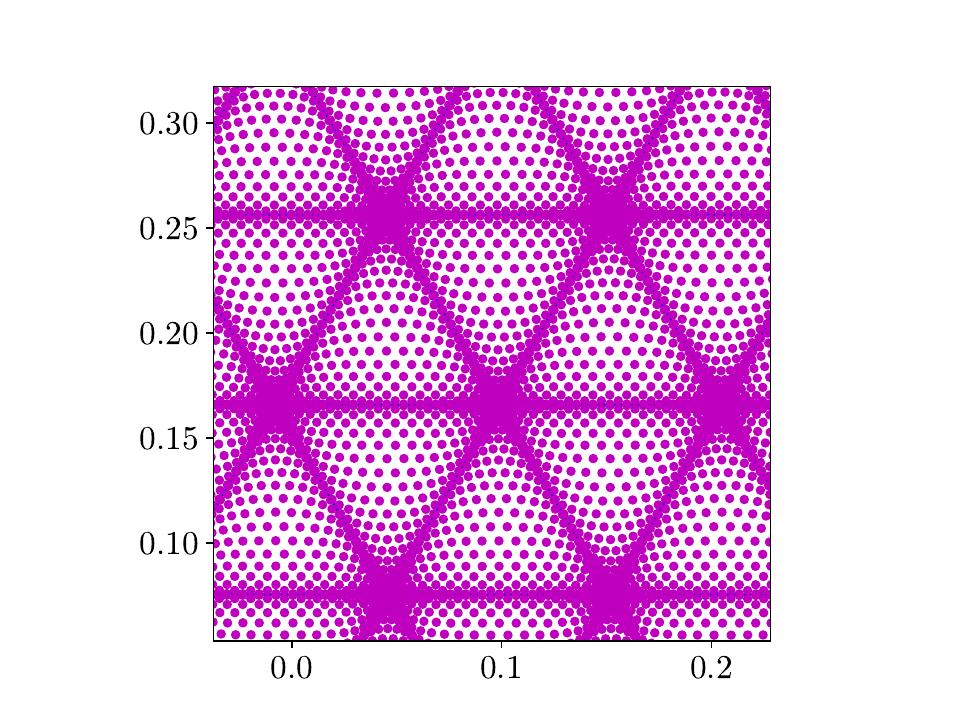}
  \caption{
      {\bf Interpolation nodes over a mesh}. Note that the nodes cluster
      around the edges. }
   \label{fig:cluster}
\end{figure}

However, it is important to note that the non-uniform potential correction
cost described above is an artifact of the discretization of the
domain (see Appendices \ref{sec:distmesh}, \ref{sec:mod_distmesh}), rather
than an intrinsic difficulty of the problem. We observe that if one instead
staggers the interpolation mesh with the quadrature mesh, both the number of
potential corrections for each target, and the number of subdivisions needed
for resolving the nearly-singular integrands, are reduced substantially.
By ``stagger'', we mean that the edges of the interpolation mesh are
maximally non-overlapping with the edges of the quadrature mesh, such that
the interpolation nodes cluster around the centroid of each element in the
quadrature mesh (see Figure \ref{fig:stag_mesh}).

A good heuristic approximation to such an interpolation mesh can be obtained
by shifting the initial guess to the meshing algorithm (the triangular
tiling, see Appendix \ref{sec:distmesh}), such that the initial guess for the
interpolation mesh becomes staggered with the initial guess for the
quadrature mesh. 

We report the performance of the interpolation, with and without the use of
a staggered mesh, in Section \ref{sec:experiment_bern}.

\begin{remark}
With the use of a staggered mesh, some target points could be
extremely close to some edges in the mesh. As is shown in Tables
\ref{tab:near_corr} and \ref{tab:self_stag}, this turns out not to be a
problem, since the number of required subdivisions increases logarithmically
with respect to the distance between the target point and the edges, and
such targets only make up a small proportion of the total target points.
\end{remark}

\section{Numerical experiments}
In this section, we illustrate the performance of the algorithm with several
numerical examples.  We implemented our algorithm in FORTRAN 77, and
compiled it using  the Intel Fortran Compiler, version 2021.5.0, with the
\texttt{-Ofast} flag. We conducted all experiments on a ThinkPad laptop, with 16 GB of
RAM and an Intel Core i7-10510U CPU. We note that, in our implementation, we
do not construct interpolants of the density function, but rather always
evaluate the density function naively. Thus, the timing results that we
present depend on the actual cost of evaluating the density functions.
Furthermore, for simplicity, we numerically reparametrize
the input parametrized curve (i.e., the boundary of the domain $\Omega$)
by arc length, which results in a somewhat costly evaluation of the
reparametrized curve. This, however, does not affect the spirit of the
experimental results that we present, i.e., the acceleration of
the computation using the techniques described in this paper.

We use the code that is publicly available in the companion code of
\cite{bremer1} for the evaluation of Koornwinder polynomials. We also
use the tables of Xiao-Gimbutas rules and Vioreanu-Rokhlin rules that are
publicly available in \cite{triasymq}. We use the FMM library published in
\cite{fmm2d} in our implementation. We make no use of the
high-performance linear algebra libraries, e.g., BLAS, LAPACK, etc.

We list the notations that appear in this section below.
\begin{itemize}
\item $h_0$: the mesh element size.
\item $\varepsilon$: the error tolerance of the far and near field
interaction computations (controlled by the precise near field geometry
analysis).
\item $N$: the order of the quadrature rule. In particular, we use the following
notations to denote the value under the special settings.
\begin{itemize}
\item $\nordf$: the order of the far field quadrature rule.
\item $\nordn$: the order of the quadrature rule used in the near field
interaction computation.
\item $\nordl$: the order of the Gauss-Legendre quadrature rule used in the
self-interaction computation (along the arc length coordinate).
\item $\nordg$: the order of the generalized Gaussian quadrature rule used in the
self-interaction computation (along the radial coordinate), in the sense
that it integrates both $\phi(r)$ and $r\log r\cdot \phi(r)$ over $[0,1]$
exactly, where $\phi(r)$ is a polynomial of order up to $\nordg$.
\item $\nords$: the order of the interpolation scheme.
\end{itemize}
\item $N_{\textit{tot}}$: the total number of discretization nodes.
\item $\ncor$: the average number of the potential corrections (including
corrections both over triangles and curved elements), for each
target location. 
\item $\nsubn$: the average number of subtriangles that one needs to
integrate over, to resolve the nearly-singular integrands in the near field
interaction computations, including both the triangle and curved element
integration domains, for each target location.
\item $\nsubl$: the total number of subdivisions along the arc length coordinate
to resolve the nearly-singular integrands in the self-interaction computations.
\item $T_{\ft}$: The time spent on far field interaction computations.
\item $T_{\nt}$: The time spent on near field interaction computations.
\item $T_{\ft+\nt}$: The total time spent on far and near field interaction computations.
\item $T_{\st}$: The time spent on self-interaction computations.
\item $T_{tot}$: The total time for the evaluation of the volume potential at all
of the discretization nodes.
\item $E_{\textit{abs}}$: the largest absolute error of the potential evaluations
at all of the interpolation nodes.
\item $\tilde E_{\textit{abs}}$: the largest absolute error of the solution to Poisson's
equation at all of the interpolation nodes.
\item $\frac{\#\textit{tgt}}{\textit{sec}}$: the number of targets that the algorithm (with
the use of precise near field geometry analysis and a staggered mesh) can
evaluate the volume potential at, per second.
\end{itemize}

We list the superscripts that appear in the notations below.
\begin{itemize}
\item $0$ (e.g., $c^0, s_n^0$): the experimental setting where the
near field is approximated naively by a ball, and the staggered mesh is not
used.
\item $+$ (e.g., $c^+, s_n^+$): the experimental setting where the
near field is approximated by the union of Bernstein ellipses, and the
staggered mesh is not used.
\item $*$ (e.g., $c^*, s_n^*$): the experimental setting where the
near field is approximated by the union of Bernstein ellipses, and the
staggered mesh is used.
\item $\textit{quad}$ (e.g., $N_{\textit{tri}}^{\textit{quad}}$, $N_{\textit{tot}}^{\textit{quad}}$): the quadrature
mesh-related information.  
\item $\textit{interp}$ (e.g., $N_{\textit{tri}}^{\textit{interp}}$,
$N_{\textit{tot}}^{\textit{interp}}$): the interpolation mesh-related information.
\end{itemize}

In Tables \ref{tab:ord_len_quad} and \ref{tab:ord_len_interp}, we tabulate
the orders and lengths of the quadrature rules used in our implementation. 

\begin{table}[h!!]
    \begin{center}
    \begin{tabular}{ccc}
    Type & $\nordf$ & Length  \\
    \midrule
    \addlinespace[.5em]
    X-G & 12 & 32 \\
    \addlinespace[.25em]
    X-G & 33 & 201 \\
    \addlinespace[.25em]
    X-G & 40 & 290  \\
    \addlinespace[.25em]
    X-G & 50 & 444  \\
    \addlinespace[.25em]
    GGQ & 8 &  8 \\
    \end{tabular}
    \caption{
    {\bf The orders and lengths of the quadrature rules}.
    We note that X-G denotes the Xiao-Gimbutas rules, and GGQ
    denotes the generalized Gaussian quadrature rule that we
    make use of in Section \ref{sec:self_int}.
    }
    \label{tab:ord_len_quad}
    \end{center}
\end{table}

\begin{table}[h!!]
    \begin{center}
    \begin{tabular}{cccc}
    $\nords$ & $\nordf$ & Length & $\kappa$\\
    \midrule
    \addlinespace[.5em]
    12 & 20 & 91 &19.2\\
    \addlinespace[.25em]
    20 & 33 & 231&194  \\
    \end{tabular}
    \caption{
    {\bf The orders, lengths and condition numbers of the Vioreanu-Rokhlin rules}.
    We denote the condition number of the interpolation matrix by $\kappa$.
    }
    \label{tab:ord_len_interp}
    \end{center}
\end{table}

\subsection{Effectiveness of the acceleration techniques}
In this section, we demonstrate the effectiveness of the acceleration
techniques described in Section \ref{sec:acceleration}. We fix $\nordn=12$,
$\nordl=16$, $\nordg=8$ (in fact, the values of $\nordl$ and $\nordg$ are
irrelevant to the experimental results presented in this Section, as we only
report the number of subdivisions). 

\subsubsection{Precise near field estimation and staggered mesh-based interpolation}
  \label{sec:experiment_bern}
In this section, we first demonstrate how well the near field is
characterized by the union of Bernstein ellipses in Figures \ref{fig:nf1},
\ref{fig:nf0}, \ref{fig:nf2}.
Additionally, we provide an illustration of two staggered
meshes in Figure \ref{fig:stag_mesh}. Then, we report the effect of the near
field geometry analysis and the use of a staggered mesh on the average number of near
field potential corrections, and the average number of the subtriangles that
one needs to integrate over, for each target, in Table \ref{tab:near_corr}.
In Table \ref{tab:self_stag}, we also report the number of subdivisions
along the arc length coordinate in the computation of self-interactions,
with and without the use of a staggered mesh.

To estimate the errors, we consider the computation of 
  \begin{align}
u(x)=\iint_{\Omega} \frac{1}{2\pi}\log(\norm{x-y})\d A_y,
  \label{for:exp_ux}
  \end{align}
where the integration domain $\Omega$ is a circle with radius $1$.
In Table \ref{tab:prob_size}, we report the size of this problem
for varioius mesh sizes $h_0$.  It can be easily shown that
$u(x)=\frac{1}{4}(\norm{x}^2-1)$.  Again, we note that the cost of
evaluating the density function is independent of the experimental results
that we present here, as we only report the number of corrections and
subtriangles that one needs to integrate over, rather than the actual
time costs.

\begin{figure}[h]
    \centering
    \begin{subfigure}{0.49\textwidth}
      \centering
      \includegraphics[width=\textwidth]{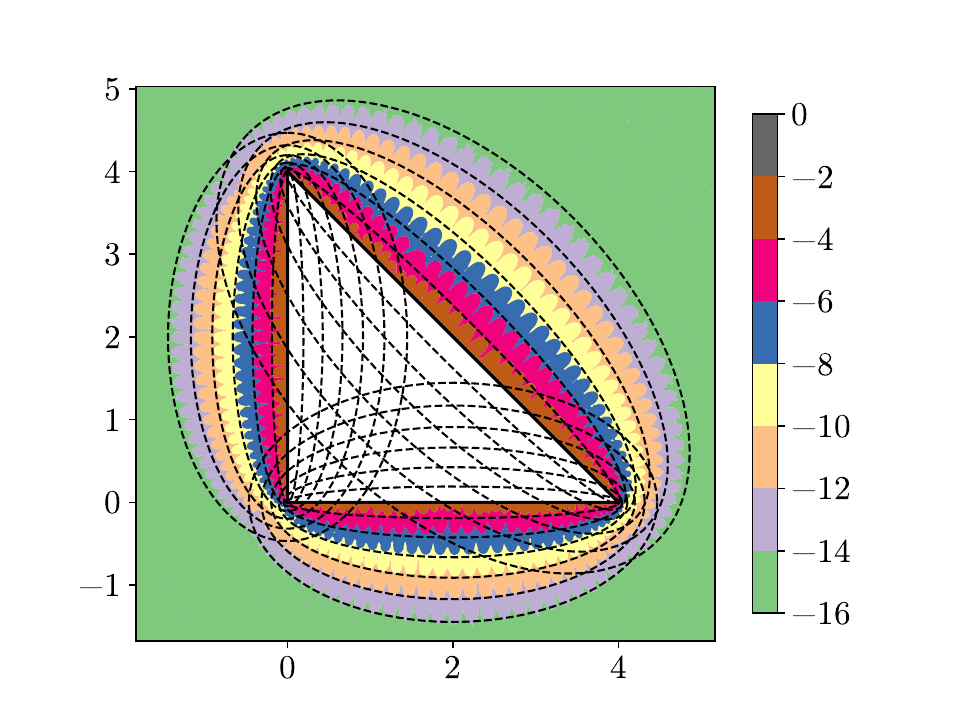}
      \caption{\label{fig:nf1_a}}
    \end{subfigure}
    \begin{subfigure}{0.49\textwidth}
      \centering
      \includegraphics[width=\textwidth]{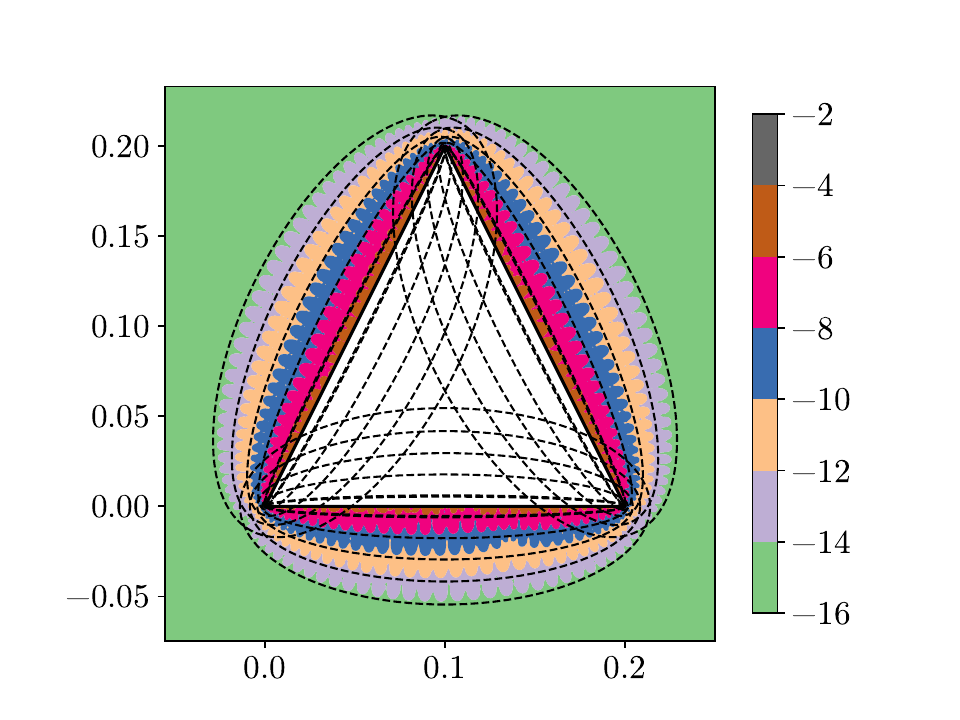}
      \caption{}
    \end{subfigure}
  \begin{subfigure}{0.49\textwidth}
    \centering
    \includegraphics[width=\textwidth]{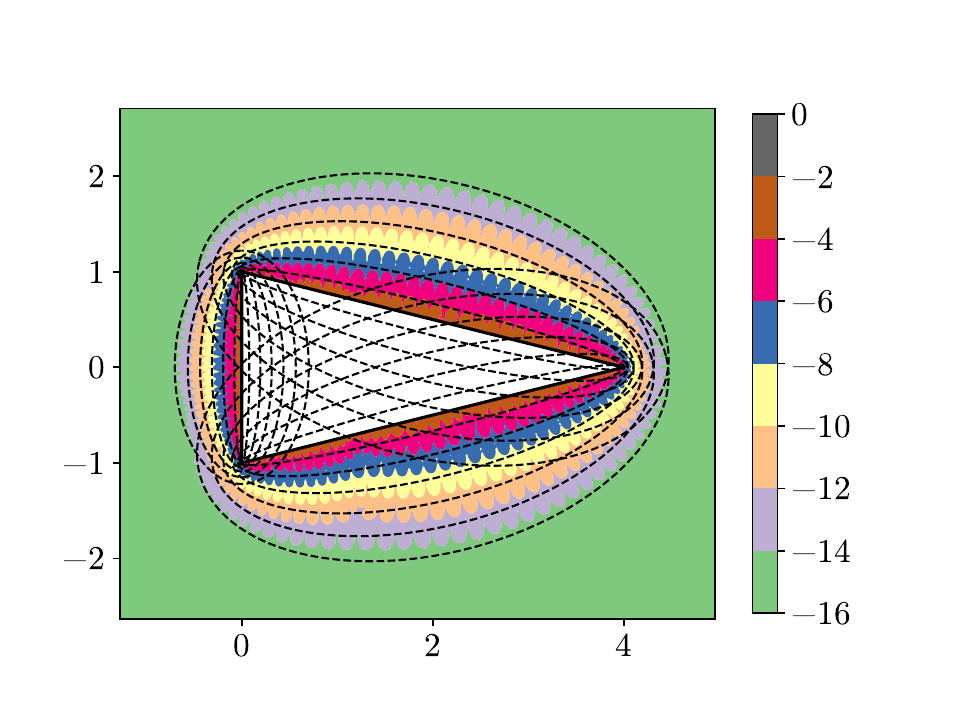}
    \caption{}
  \end{subfigure}
  \begin{subfigure}{0.49\textwidth}
    \centering
    \includegraphics[width=\textwidth]{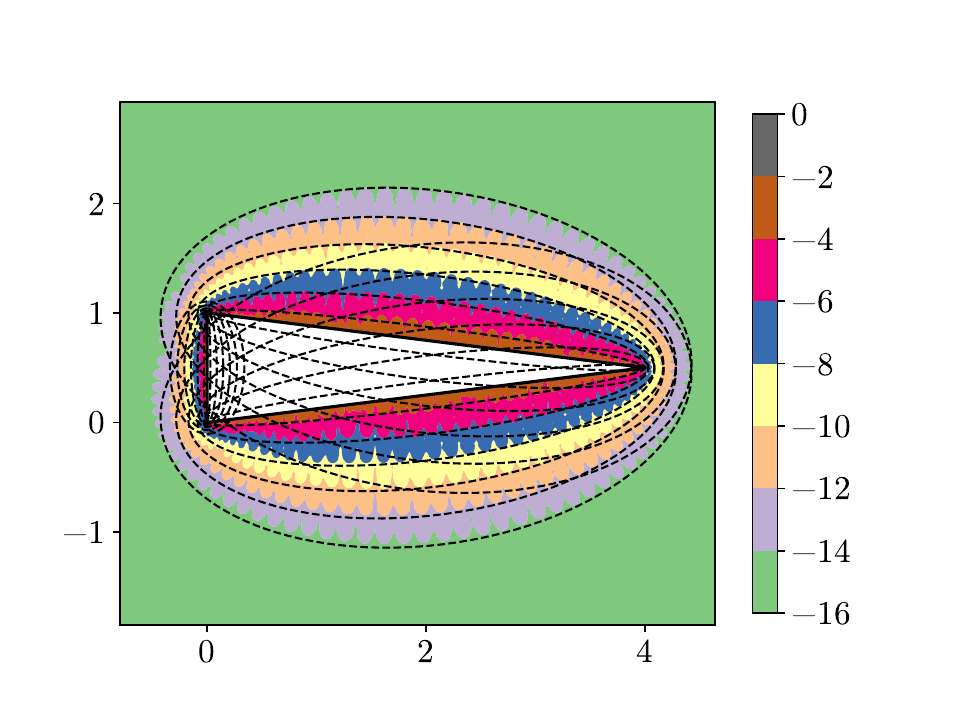}
    \caption{\label{fig:nf1_d}}
  \end{subfigure}

  \caption{{\bf The contour plot of $\log_{10}E_{abs}$
  with the Xiao-Gimbutas quadrature rule of order 40, and the
  estimation of the near field}. We select the density $f$ to be $1$, and
  the parameter $C_N$ to be $2$ in all these plots.}
  \label{fig:nf1}
\end{figure}

\begin{figure}[h]
    \centering
    \begin{subfigure}{0.49\textwidth}
      \centering
      \includegraphics[width=\textwidth]{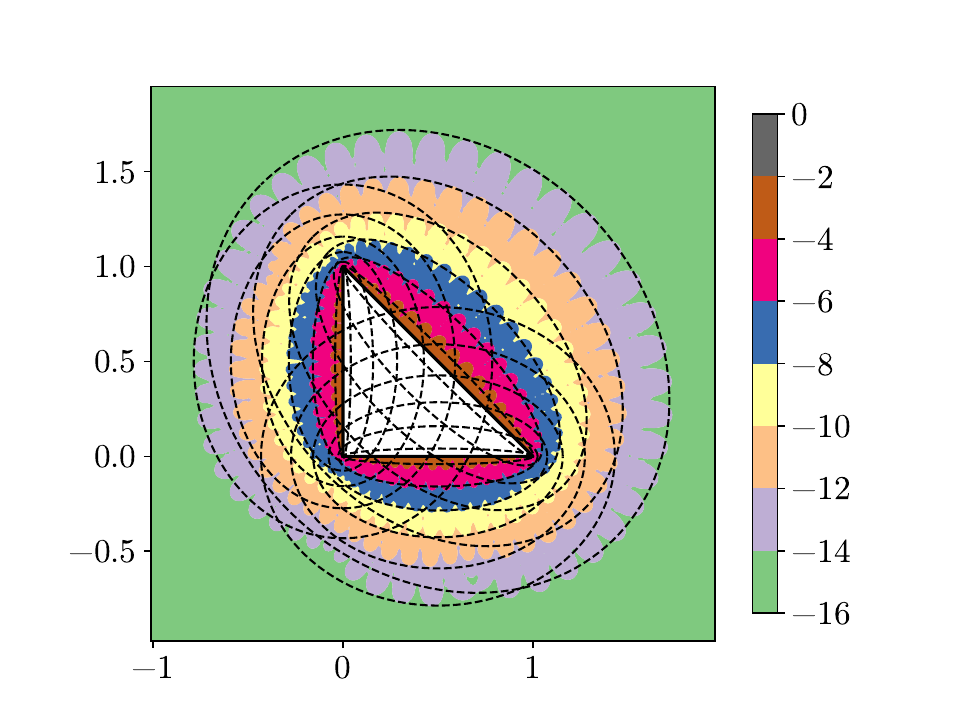}
      \caption{\label{fig:nf0_a}}
    \end{subfigure}
    \begin{subfigure}{0.49\textwidth}
      \centering
      \includegraphics[width=\textwidth]{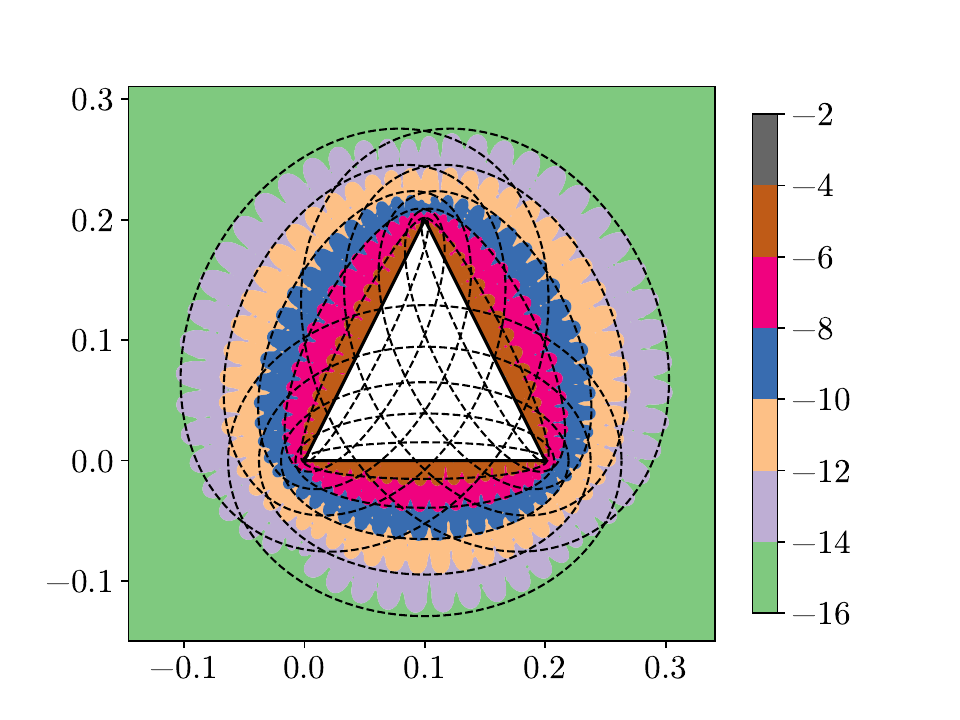}
      \caption{}
    \end{subfigure}

  \caption{{\bf The contour plot of $\log_{10}E_{abs}$
  with the Xiao-Gimbutas quadrature rule of order 20, and the
  estimation of the near field}. We select the density $f$ to be $1$, and
  the parameter $C_N$ to be $6.8$ in all these plots. One can observe that
  when the error tolerance is low (say, $E_{abs}<10^{-14}$), the true near
  field is almost a ball. 
  }
  \label{fig:nf0}
\end{figure}

\begin{figure}[h]
    \centering
    \begin{subfigure}{0.49\textwidth}
      \centering
      \includegraphics[width=\textwidth]{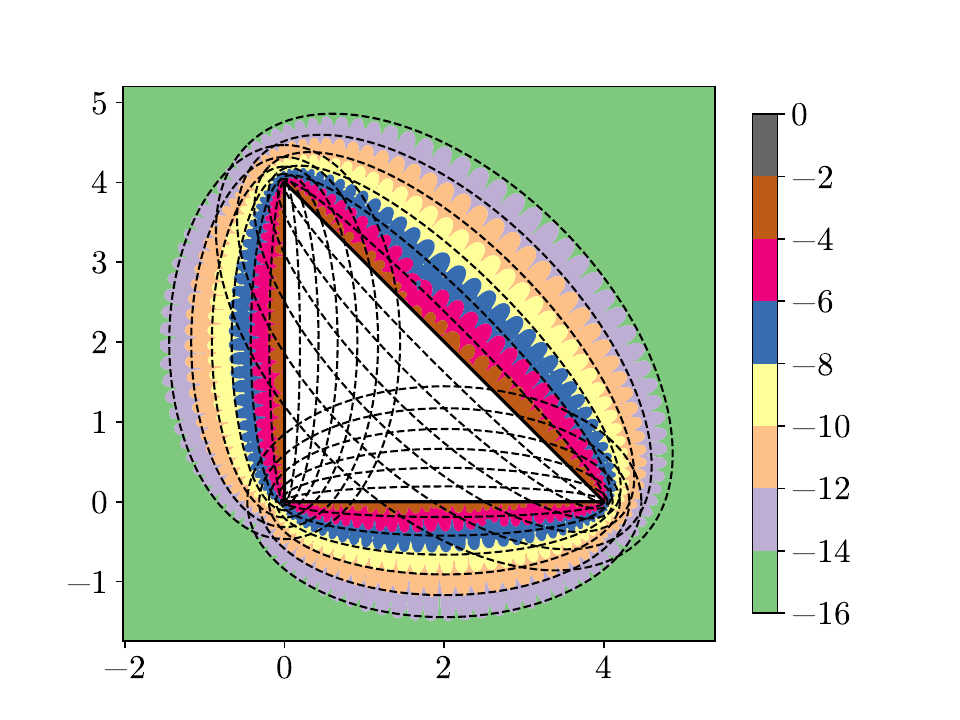}
      \caption{$f(x,y)=\sin(x+2y)$ \label{fig:nf2_a}}
    \end{subfigure}
    \begin{subfigure}{0.49\textwidth}
      \centering
      \includegraphics[width=\textwidth]{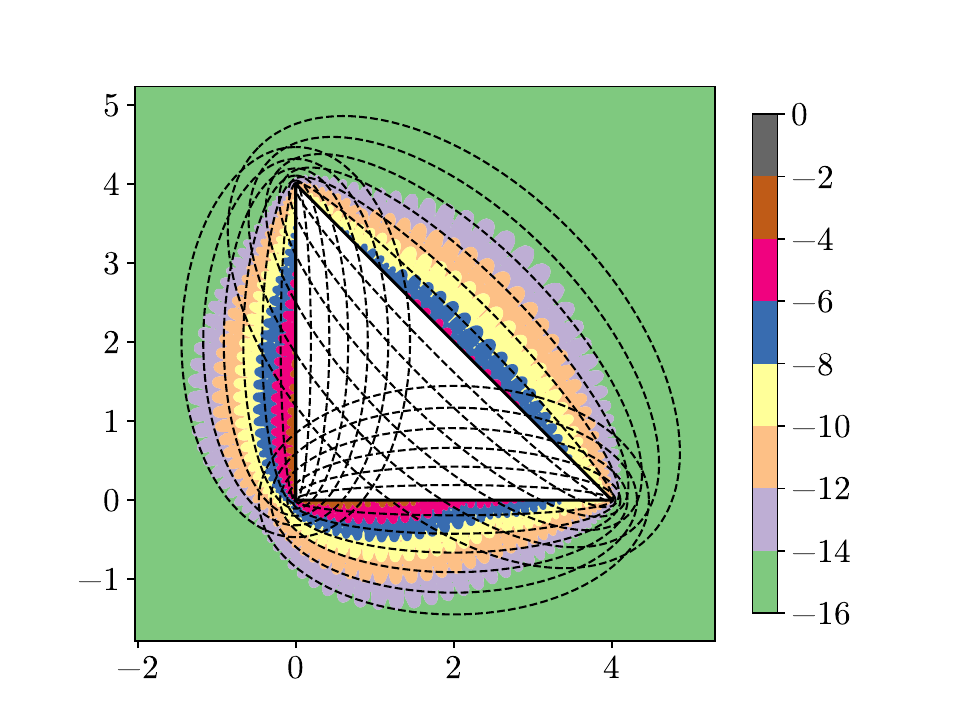}
      \caption{$f(x,y)=\exp(-x^2-y^2)$ \label{fig:nf2_b}}
    \end{subfigure}
  \caption{{\bf The contour plot of $\log_{10}E_{abs}$
  with the Xiao-Gimbutas quadrature rule of order 40, and the
  estimation of the near field}. We select the parameter $C_N$ to be $2$
  in all these plots. Comparing Figure \ref{fig:nf2_a}
  with Figure \ref{fig:nf1_a}, one can observe that the near field
  only becomes slightly larger when the density function changes from $1$ to
  $\sin(x+2y)$. In general, as long as the density function can
  be well-resolved by the selected quadrature rule, the shape of the near
  field is nearly the same as the shape of the near field with density
  function $1$. In Figure \ref{fig:nf2_b}, one can observe that the actual near
  field is much smaller than the predicted near field, especially for
  the region that is away from the origin $(0,0)$. This is because the
  density function $\exp(-x^2-y^2)$ decays exponentially as $|x|, |y|$
  increase, which cancels out the logarithmic singularity that is away
  from the origin. Thus, such a behavior is expected.
  }
  \label{fig:nf2}
\end{figure}

\begin{figure}[h]
    \centering
    \includegraphics[width=0.75\textwidth]{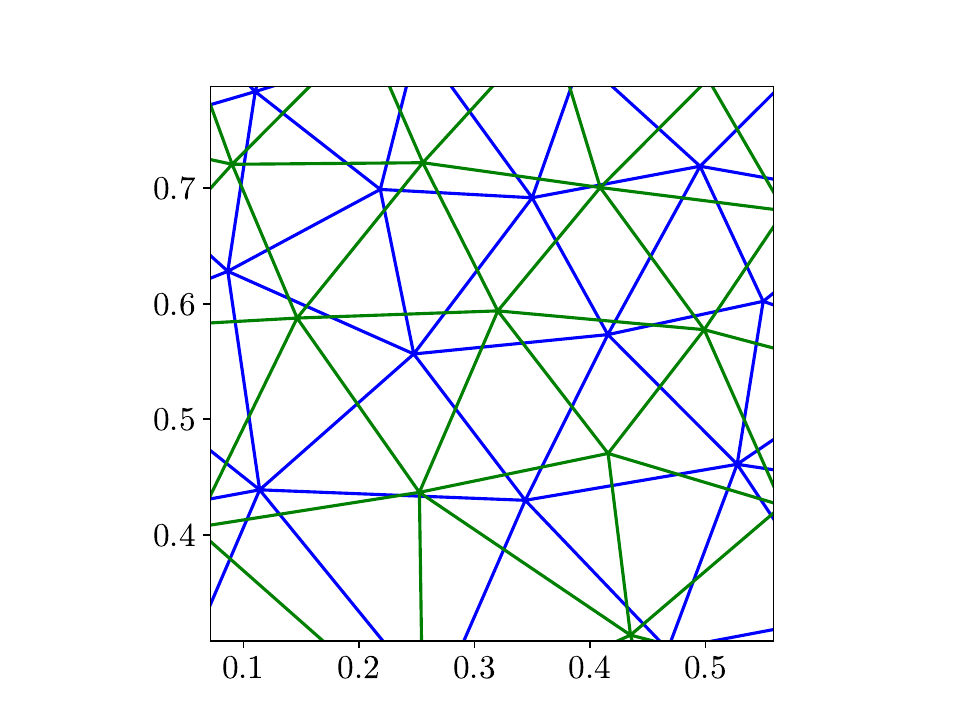}
  \caption{
      {\bf The staggered mesh under a microscope}. }
   \label{fig:stag_mesh}
\end{figure}

\begin{table}[h!!]
    \begin{center}
    \begin{tabular}{ccc}
    $h_0$ & $N_{\textit{tri}}$ &  $N_{\textit{tot}}$  \\
    \midrule
    \addlinespace[.5em]
    0.2 & 143 &  33033 \\
    \addlinespace[.25em]
    0.1 & 657 &  151767 \\
    \addlinespace[.25em]
    0.05 & 2766 & 638946 \\
    \end{tabular}
    \caption{
    {\bf The total number of mesh elements for different mesh sizes $h_0$
    with $N_s=20$}.
    We also report the order of the Vioreanu-Rokhlin rules, and the 
    total number of the discretization nodes. 
    }
    \label{tab:prob_size}
    \end{center}
\end{table}

\begin{table}[h!!]
    \begin{center}
    \begin{tabular}{c cccc cccc c}
    $\varepsilon$ & $\ncor^{0}$&
    $\ncor^{+}$& $\ncor^{*}$&$\ncor^*/\ncor^0$ &$\nsubn^{0}$ &
    $\nsubn^{+}$ & $\nsubn^{*}$ & $\nsubn^*/\nsubn^0$ &  $E_{\textit{abs}}^*$ \\
    \midrule
    \addlinespace[.5em]
      $10^{-8}$   & 1.46 & 0.99 & 0.50 & 33.9\% &15.7 & 9.67 & 4.19 &26.8\%&2.06\e{-8}\\
    \addlinespace[.25em]
     $10^{-10}$ &   2.32 & 1.49 & 0.91 & 39.2\% &30.7 & 20.2  &10.3 &33.7\%&1.78\e{-10}\\
    \addlinespace[.25em]
     $10^{-12}$ &   2.94 & 2.03 & 1.39 & 47.3\% &55.0 & 36.9  &22.4 &40.8\%&1.61\e{-12}\\
    \addlinespace[.25em]
      $10^{-14}$ &   3.49 & 2.64 & 1.98 & 56.5\% & 103 &72.1 &47.1 &45.3\%&2.47\e{-14}\\
    \end{tabular}
    \caption{
    {\bf The average number of near field potential corrections, and the
    average number of subtriangles that need to be integrated over to resolve
    the nearly-singular integrand during the computation of near field
    interactions, for each target, with and without the precise near field
    analysis and the use of a staggered mesh.} Note that we fine-tuned 
    the parameters, such that $E_{\textit{abs}}^0$ and $E_{\textit{abs}}^+$
    are around the same size as $E_{\textit{abs}}^*$.
    }
    \label{tab:near_corr}
    \end{center}
\end{table}

\begin{table}[h!!]
    \begin{center}
    \begin{tabular}{c ccc }
    $h_0$ & $\nsubl^{0}$ & $\nsubl^{*}$ & $\nsubl^*/\nsubl^0$ \\
    \midrule
    \addlinespace[.5em]
      $0.2$   &18.8 & 16.5 &87.6\%\\
    \addlinespace[.25em]
      $0.1$   &18.8 & 16.2 &86.1\%\\
    \addlinespace[.25em]
      $0.05$   &18.7 & 16.0 &85.7\%\\
   \end{tabular}
    \caption{
    {\bf Average number of self-interaction subdivisions with and without
    the use of a staggered mesh. } One can observe that the use of a
    staggered mesh reduces the amount of computation of the
    self-interactions by $15\%$.
    }
    \label{tab:self_stag}
    \end{center}
\end{table}

In fact, the use of the near field geometry analysis together with a staggered mesh is more
powerful than Table \ref{tab:near_corr} indicates, for the following
reason. To make a fair comparison with the standard way of computing the
near interactions (i.e., using the Vioreanu-Rokhlin rule over a single mesh),
we are bound to a fixed pair of quadrature and interpolation
orders. Comparing Figure \ref{fig:nf1} with Figure \ref{fig:nf0}, it is easy
to see that the experimental results in Table \ref{tab:near_corr} will be even
more impressive if we use a higher-order quadrature rule. We exploit
this fact in Sections \ref{sec:experiment_trade} and
\ref{sec:experiment_pot}.

\pagebreak

\subsubsection{Trade-off between the far and near field interaction computations}
\label{sec:experiment_trade}
In this section, we demonstrate the effectiveness of the offloading
technique for reducing the total amount of time spent on the 
near field interaction computations in Figure \ref{fig:offload}. In our
examples, we consider the computation of
  \begin{align}
u(x)=\iint_{\Omega} \frac{1}{2\pi}\log(\norm{x-y}) f(y)\d A_y,
  \end{align}
where the integration domain $\Omega$ is a circle with radius $1$,
and the density $f$ is
  \begin{align}
f(x_0,y_0)=&\,4e^{-(x_0+1.6)^2-(y_0+0.2)^2}\cdot(x_0^2+y_0^2+3.2x_0+0.4y_0+1.6) + \notag\\
&\,4e^{-(x_0-0.2)^2-(y_0-1)^2}\cdot (x_0^2+y_0^2-0.4x_0-2y_0+0.04).
  \end{align}

\begin{figure}[h]
    \centering
    \begin{subfigure}{0.49\textwidth}
      \centering
      \includegraphics[width=\textwidth]{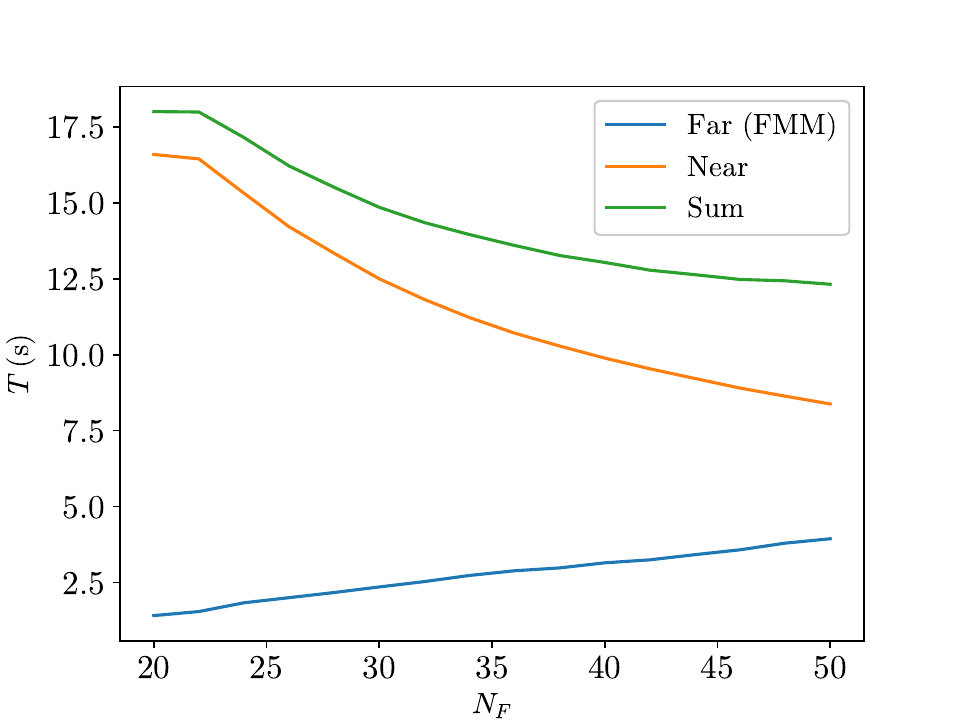}
      \caption{$\varepsilon=10^{-14}, N_s=20$}
    \end{subfigure}
    \begin{subfigure}{0.49\textwidth}
      \centering
      \includegraphics[width=\textwidth]{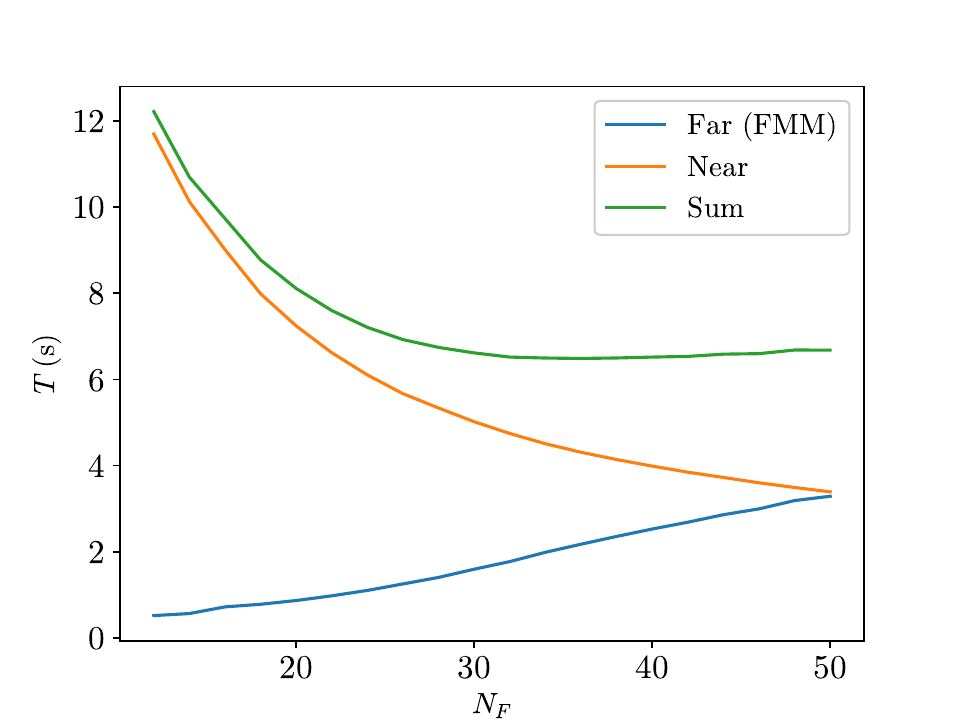}
      \caption{$\varepsilon=10^{-14}, N_s=12$}
    \end{subfigure}
  \begin{subfigure}{0.49\textwidth}
    \centering
    \includegraphics[width=\textwidth]{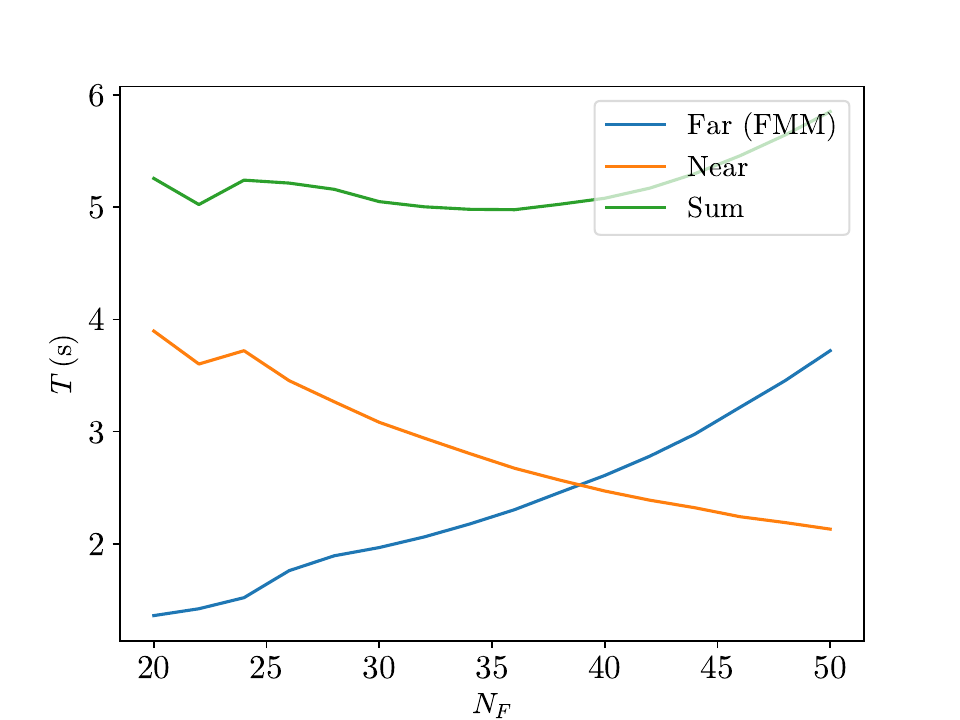}
    \caption{$\varepsilon=10^{-10}, N_s=20$}
  \end{subfigure}
  \begin{subfigure}{0.49\textwidth}
    \centering
    \includegraphics[width=\textwidth]{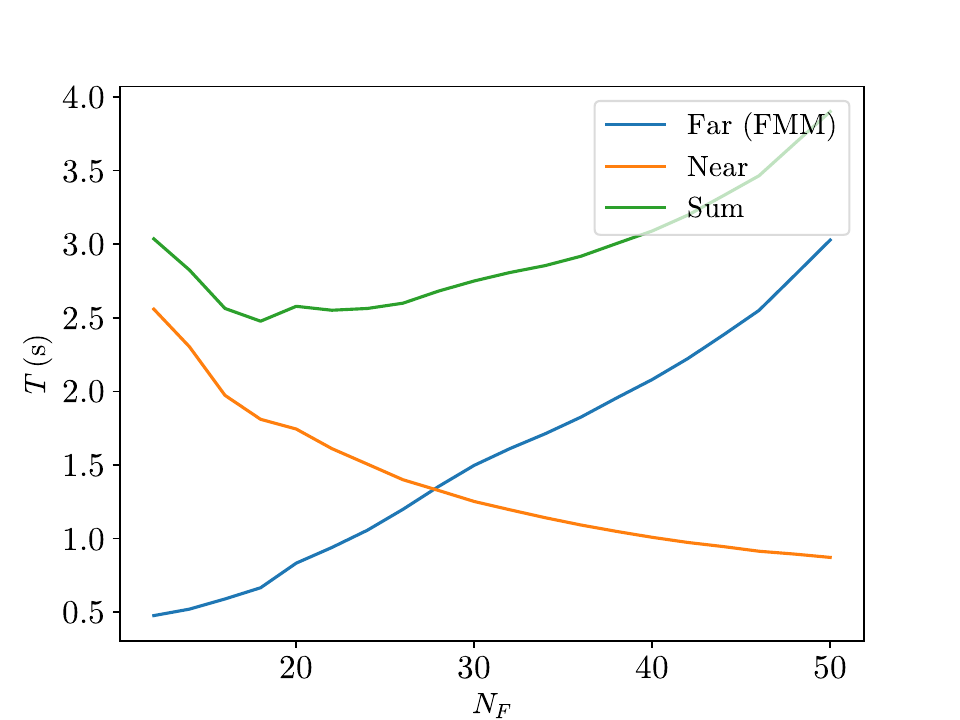}
    \caption{$\varepsilon=10^{-10}, N_s=12$}
  \end{subfigure}

  \caption{{\bf The offloading technique applied to volume potential
  evaluations with different interpolation orders and error tolerances,
  without parallelization}. 
  We fix the element size to be $h_0=0.1$ in all cases. If the FMM computations
  are parallelized, the time spent on the computation of far field
  interactions can be reduced dramatically, depending on the number of
  cores.}
  \label{fig:offload}
\end{figure}

\pagebreak

\subsection{Computation of the volume potential and Poisson's equation} 
\label{sec:experiment_pot}
In this section, we report the accuracy (implicitly, by reporting the
accuracy of the solution to Poisson's equation) and speed of the computation
of the volume potential
  \begin{align}
u(x)=\iint_{\Omega} \frac{1}{2\pi}\log(\norm{x-y})f(y)\d A_y,
  \end{align}
where the density function
  \begin{align}
f(x_0,y_0)=&\,4e^{-(x_0+1.6)^2-(y_0+0.2)^2}\cdot(x_0^2+y_0^2+3.2x_0+0.4y_0+1.6) + \notag\\
&\,4e^{-(x_0-0.2)^2-(y_0-1)^2}\cdot (x_0^2+y_0^2-0.4x_0-2y_0+0.04),
  \end{align}
and the domain $\Omega$ is a wobbly ellipse, as is displayed in Figure
\ref{fig:wob_domain_mesh}. The sizes of the our experiments are
presented in Table \ref{tab:poi_prob_size}. We compare the performance of
the algorithm, with and without the use of precise near field geometry
analysis and a staggered mesh, in Tables \ref{tab:poi1} and \ref{tab:poi2}.
Additionally, we solve the Poisson's equation
  \begin{align}
\hspace*{-0em}\nabla^2 \varphi = &\,f(x_0,y_0) \text{ in } \Omega, \notag\\
\hspace*{-0em}\varphi=&\,g(x_0,y_0) \text{ on } \partial\Omega,
\label{for:poi}
  \end{align}
where 
\begin{align}
g(x_0,y_0)=\exp(-(x_0+1.6)^2-(y_0+0.2)^2)+\exp(-(x_0-0.2)^2-(y_0-1)^2),
\end{align}
and we present the error heat map in Figure \ref{fig:poi_heat}.  Below, we
sketch the algorithm for solving Poisson's equation with the use of
the volume potential.

Since the volume potential 
  \begin{align}
u(x)=\frac{1}{2\pi}\iint_\Omega \log\norm{x-y} f(y) \d A_y 
\label{for:volpot_poi}
  \end{align}
satisfies
  \begin{align}
\nabla^2 u = f \text{ in } \Omega,
  \end{align}
provided that $u^h:\Omega\to\R$ solves Laplace's equation
  \begin{align}
\nabla^2 u^h = 0 \text{ in } \Omega, \notag\\
u^h=g-u \text{ on } \partial\Omega, 
  \label{for:lap}
  \end{align}
we have that $\varphi:=u^h+u$ satisfies the given Poisson's equation. 
In our implementation, we compute $u(x)|_{\partial\Omega}$ through
interpolation (see Remark \ref{rem:bndry_interp}), from which it follows
that the condition number of the interpolation matrix affects the accuracy
of our computational results (see \cite{vior,triasymq}). Then, we find the
solution to the Laplace equation (\ref{for:lap}) by the boundary integral
equation method \cite{fds}.  Finally, we note that the true solution
$\varphi$ to this Poisson's equation equals $g$.

\begin{figure}[h]
    \centering
    \includegraphics[width=0.75\textwidth]{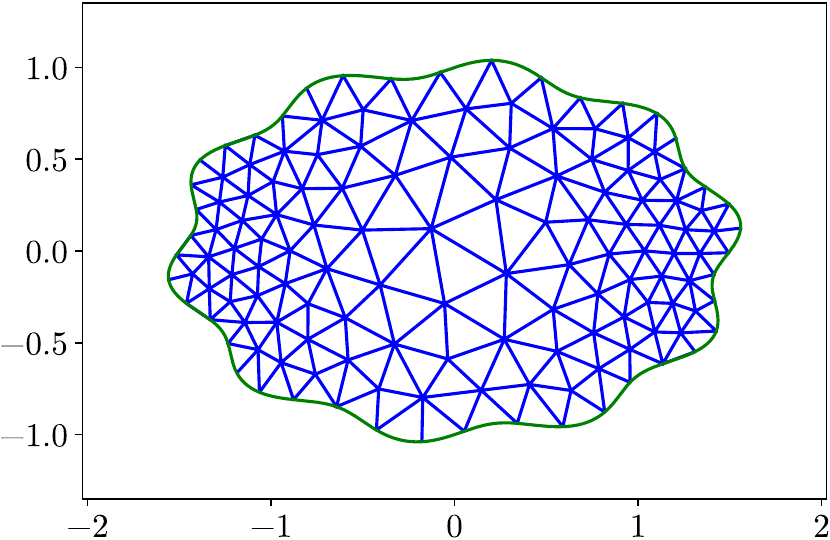}

  \caption{
      {\bf The wobbly ellipse domain $\Omega$ discretized by an unstructured mesh
      with $h_0=0.2$}. }
   \label{fig:wob_domain_mesh}

\end{figure}

\begin{table}[h!!]
    \begin{center}
    \begin{tabular}{ccccc}
    $h_0$ & $N_{\textit{tri}}^{\textit{quad}}$ &
    $N_{\textit{tri}}^{\textit{interp}}$ &
    $N_{\textit{tot}}^{\textit{quad}}$ &
    $N_{\textit{tot}}^{\textit{interp}}$ \\
    \midrule
    \addlinespace[.5em]
    0.2 & 231 & 225&  102564 & 51975 \\
    \addlinespace[.25em]
    0.1 & 1001& 997 & 444444 & 230307\\
    \addlinespace[.25em]
    0.05 &4183& 4172& 1857252 & 963732\\
    \end{tabular}
    \caption{
    {\bf The total number of mesh elements, quadrature and interpolation
    nodes for different mesh sizes $h_0$}. In this table, we set
    $N_f=50,N_s=20$.
    }
    \label{tab:poi_prob_size}
    \end{center}
\end{table}

\begin{table}[h!!]
    \begin{center}
    \begin{tabular}{cc ccc cc cc c c}
    $\varepsilon$ & $h_0$ & $T_{\ft+\nt}^{0}$&
    $T_{\ft+\nt}^*$ & $\frac{T_{\ft+\nt}^{*}}{T_{\ft+\nt}^{0}}$&$T_{\st}^{0}$ &
    $T_{\st}^{*}$  & $T_{\textit{tot}}^0$ & 
    $T_{\textit{tot}}^*$ &   $\tilde E_{\textit{abs}}^{*}$ & $\frac{\#\textit{tgt}}{\textit{sec}}$\\
    \midrule
    \addlinespace[.5em]
    $10^{-10}$&  $0.2$   & 4.14 & 1.98 & 47.8\% &  2.64 & 2.39 & 7.08 & 4.72&
    2.66\e{-9} & 1.10\e{4}\\
    \addlinespace[.25em]
    &  $0.1$   & 16.0 & 7.93 & 49.4\% &  12.3 & 11.0 & 30.3 & 21.0&
    2.42\e{-9} & 1.10\e{4}\\
    \addlinespace[.25em]
    &  $0.05$   & 60.5 & 29.6 & 49.0\% &  50.9 & 45.3 & 85.9 & 46.9&
    4.69\e{-9} & 1.12\e{4}\\
    \addlinespace[.25em]
    $10^{-14}$&  $0.2$   & 9.69 & 5.96 & 61.5\% &  2.65 & 2.72 & 12.6 & 9.06&
    2.42\e{-12} & 5.73\e{3}\\
    \addlinespace[.25em]
    &  $0.1$   & 36.3 & 19.3 & 53.2\% &  12.5 & 11.4 & 50.6 & 32.7&
    4.13\e{-13} & 7.05\e{3}\\
    \addlinespace[.25em]
    &  $0.05$   & 136 & 67.2 & 49.3\% &  70.7 & 63.4 & 217 & 141&
    4.54\e{-13} & 6.84\e{3}\\
    \addlinespace[.25em]
    \end{tabular}
    \caption{
    {\bf Comparisons of the computational time of the volume potential
    evaluation with $20$th order interpolation of the solution, without
    parallelization}. When $\varepsilon=10^{-10}$, we set $\nordf=40$,
    $\nordn=12$, $\nordl=10$, $\nordg=8$; when $\varepsilon=10^{-14}$, we
    set $\nordf=50$, $\nordn=12$, $\nordl=10$ (except that $\nordl=14$ when
    $h_0=0.05$), $\nordg=8$. In all of these experiments, the Vioreanu-Rokhlin
    rule of order $20$ is used in the naive case indicated by the
    superscript $0$. It is important to note that we only report the error of
    the approximation to the solution to Poisson's equation here, as the
    analytic solution for the volume potential is not available. However, we
    note that the actual accuracy of the volume potential evaluations is
    higher than the one shown in the $\tilde E_{abs}^{*}$ column, due to the use
    of not perfectly-conditioned interpolation matrices. Note that we
    fine-tuned the parameters, such that $E_{\textit{abs}}^0$ is
    around the same size as $E_{\textit{abs}}^*$.
    }
    \label{tab:poi1}
    \end{center}
\end{table}

\begin{table}[h!!]
    \begin{center}
    \begin{tabular}{cc ccc cc}
    $\varepsilon$ & $h_0$ & $T_{\ft}^{0}$& $T_{\ft}^*$ &  $T_{\nt}^{0}$&
    $T_{\nt}^*$ &$\frac{T_{\nt}^{*}}{T_{\nt}^{0}}$\\
    \midrule
    \addlinespace[.5em]
    $10^{-10}$&  $0.2$   & 0.55 & 0.84 & 3.59 &  1.15 & 31.9\%\\
    \addlinespace[.25em]
    &  $0.1$   & 2.75 & 4.10 & 13.3 &  3.83 &28.8\%\\
    \addlinespace[.25em]
    &  $0.05$   & 11.7 & 17.3 & 48.8 & 12.3 &25.2\%\\
    \addlinespace[.25em]
    $10^{-14}$&  $0.2$   & 0.63 &1.47 & 9.07 &  4.49 & 49.5\%\\
    \addlinespace[.25em]
    &  $0.1$   & 3.15 & 6.39 & 33.1 &  12.9 & 39.0\%\\
    \addlinespace[.25em]
    &  $0.05$   & 13.3 & 27.0 & 123 & 40.2 & 32.6\%\\
    \addlinespace[.25em]
    \end{tabular}
    \caption{
    {\bf Comparisons of the computational time of the far and near field
    interactions with $20$th order interpolation of the solution, without
    parallelization}. The setting of the experiments shown in this table is
    the same as the ones shown in Table \ref{tab:poi1}. It is important to
    note that, when one parallelizes the FMM computations, $T_{\ft}$
    becomes negligible as the number of cores increases, from which it
    follows that the computational time, with
    the use of near field geometry analysis and a staggered mesh, is roughly
    equal to the values in the $\frac{T_{\nt}^{*}}{T_{\nt}^{0}}$ column
    times the naive computational time.
    }
    \label{tab:poi2}
    \end{center}
\end{table}

\begin{figure}[h]
    \centering
    \begin{subfigure}{0.49\textwidth}
      \centering
      \includegraphics[width=\textwidth]{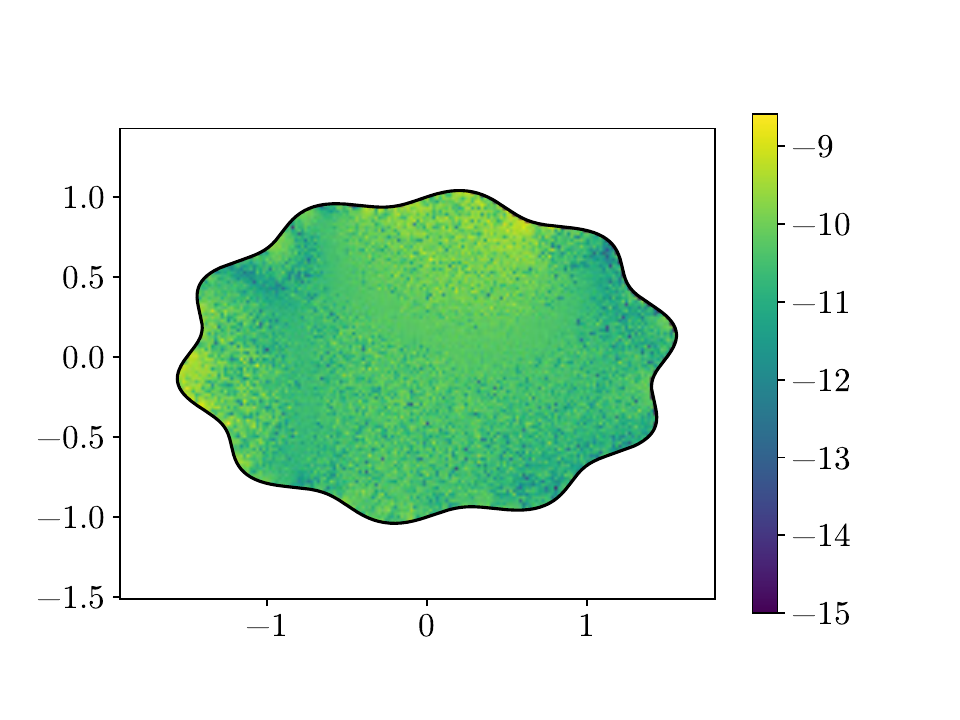}
      \caption{$\varepsilon=10^{-10}$}
    \end{subfigure}
    \begin{subfigure}{0.49\textwidth}
      \centering
      \includegraphics[width=\textwidth]{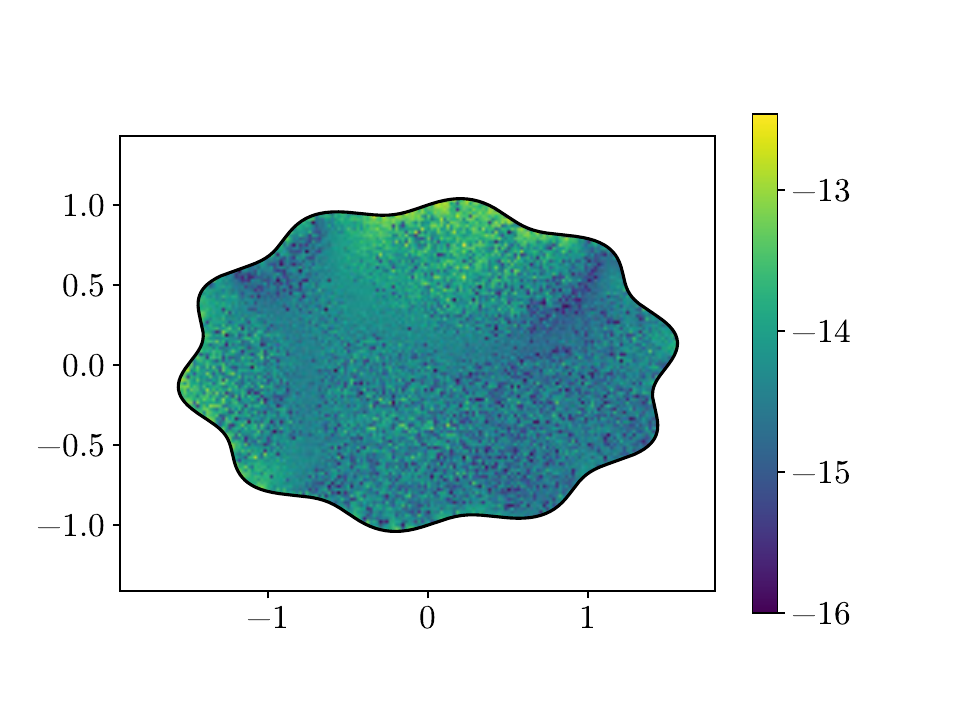}
      \caption{$\varepsilon=10^{-14}$}
    \end{subfigure}

  \caption{{\bf The heat maps of $\log_{10}\tilde E_{abs}$ when $h_0=0.05$}. 
  The results are computed by evaluating the interpolants.}
  \label{fig:poi_heat}
\end{figure}

\pagebreak

\section{Conclusions and further directions}
In this paper, we present three complementary techniques for accelerating
potential calculations over unstructured meshes, as well as a robust and
extensible framework for the evaluation and interpolation of 2-D volume
potentials over complicated geometries. With the use of precise near field
geometry analysis, we show that one can eliminate all of the unnecessary
near field potential computations. By introducing the use of a staggered
mesh, we further show that the number of
interpolation nodes at which the near field and self-interactions are
costly to evaluate is reduced dramatically.  These two observations 
facilitate the offloading technique, which transforms the
expensive and unstructured near field interaction computation into the
highly-optimized parallel FMM computation.  
Of the methods described in this paper, we believe the offloading
technique to be the most general, since we expect that the near field
will become vanishingly small when the order of the far field quadrature
rule is high, for other kernels, geometries and in higher dimensions. The
offloading technique is one of the few methods we are aware of for which
the use of extremely high-order quadrature rules is essential.

In the following sections, we discuss the generalizations and extensions of
the techniques and the volume potential evaluation algorithm presented in
this paper.
\subsection{Precise near field geometry analysis for different kernels and domains}
Although the near field geometry analysis is only applied to the 2-D
Newtonian potential in this paper, the same analysis can be trivially
generalized to the 2-D Helmholtz volume potential, as its kernel has the same
type of logarithmic singularity. The same approach to analyzing the near
field can be applied to, for example, quadrilateral elements in 2-D, and
more complicated elements in 3-D, although we expect the near field
geometry analysis in 3-D to be substantially more involved.

Since surfaces are often represented by a collection of mappings from 2-D
domains, and surfaces are often discretized by meshing these 2-D domains
using unstructured meshes, the generalization of the near field geometry
analysis to the on-surface evaluation of surface potentials is similar
to the near field analysis presented in this paper, except that
the kernel becomes more singular, and the effect of the mapping on
the near field must be accounted for.

\subsection{The offloading technique for computing surface and 3-D volume potentials}
The offloading technique clearly can be generalized to the computation of
surface and 3-D volume potentials, provided that high-order quadrature rules
are available. In practice, high-order quadrature rules for tetrahedra or
cubes are presently not available.  For example, the highest order
quadrature rule for tetrahedra reported in \cite{xiao} is 15, which is far
from enough for the offloading technique to be effective.

\subsection{The staggered mesh for surfaces and 3-D volumes}
Our heuristic way of generating the staggered mesh, described in Section
\ref{sec:stagger}, can be trivially generalized to the surface mesh and 3-D
volume mesh case. Furthermore, the use of a staggered mesh can
similarly accelerate the potential interpolation in these cases.

\subsection{Accelerating near and self-interaction computations by 
specialized quadrature rules}
Our focus in this paper is to accelerate the far and near field
interaction computations. The self-interaction computations are minimally 
optimized. After applying all of the optimizations proposed in this paper,
one can observe from Table \ref{tab:poi1} that the self-interaction
evaluation cost becomes a bottleneck of the algorithm. As is noted in Remark
\ref{rem:self_precomp}, this cost can be reduced dramatically by
precomputing a large number of specialized quadrature rules.  

In the computation of near field interactions, we use the most naive scheme,
i.e., adaptive subdivision, for resolving the nearly-singular integrand, and
we anticipate that the cost can be further reduced with the use of more
advanced methods (see, for example, \cite{anderson}). The techniques
proposed in this paper are compatible with other schemes for the computation
of near field and self-interactions.

\section{Acknowledgements}
We sincerely thank James Bremer for his helpful advice and for our
informative conversations.

\appendix

\section{Appendix: Geometric algorithms}
\label{sec:append}
In this appendix, we provide a description of all of the geometry processing
algorithms used in this paper.

\subsection{Quadtree}
  \label{sec:quadtree}
A quadtree is a tree data structure used to efficiently store points in a
two-dimensional space. More specifically, it partitions  the domain into boxes
by recursively subdividing each box into four sub-boxes until
each leaf box contains no more than $m$ points, where $m$ is a user-specified 
number. Given a set of $n$ points that are uniformly distributed,
it takes $\O(n\log n)$ operations to construct a quadtree for these points.
After the quadtree is constructed, it takes $\O(\log n)$ operations to find
all of the leaf boxes that intersect a given rectangle. We refer readers to
Chapter 37 of \cite{qtree} for a detailed introduction to the 
quadtree data structure.

\subsection{Construction of a signed distance function from
a set of parametrized curves}
  \label{sec:sdf}

In this section, we first introduce the concept of signed distance
function (SDF), and then describe an efficient algorithm for computing the
SDF of a geometry with boundaries described by a set of closed curves
  \begin{align}
\bigl\{\gamma_i(s)\bigr\}_i,\,\, \gamma_i:[0,L_i]\to\R^2, \,\, i=1,2,\dots,N,
  \end{align}
where $\gamma_i$ is a unit-speed parameterization, and $L_i$ is the total
arc length of $\gamma_i$. Below, we give the formal definition of a signed
distance function.

\begin{definition}[Signed distance function]
Given a geometry $\Omega$ with boundaries described by a set of closed curves
$\{\gamma_i(s)\}$, the signed distance function $h:\R^2\to \R$ determines 
the distance of a given point $x$ from the boundary of $\Omega$, with the sign
indicates whether $x$ is inside $\Omega$ or not. By convention, $h(x)$ is positive
for $x\in \Omega$, and negative for $x\in\Omega^c$.
\end{definition}

Our algorithm for converting the parameterized boundary curves into an SDF is
outlined as follows. We begin with the following precomputation:

\begin{enumerate}
\item (Sampling) Generate equidistant sampling points over the boundaries, where 
the number of total sampling points depends on the required accuracy.
In addition, we store the corresponding curve parameter for each sampling
point.
\item (Quadtree) Create a quadtree data structure for the sampling points
(see Appendix \ref{sec:quadtree}).
\end{enumerate}

Then, we evaluate the signed distance function at any given target location $x$
by the following steps.
\begin{enumerate}
\item Create a rectangle centered at $x$ with appropriate side lengths.
\item Query all of the sampling points that are inside the rectangle
by exploiting the quadtree data structure.
If no points are captured by the rectangle, increase the size of the
rectangle and perform the query step again.
\item Loop through all of the captured sampling points, and find the point $x'$
that is closest to the target point $x$.
\item (Optional) Do a few iterations of Newton's method using $x'$ as the
initial guess to get a highly accurate approximation to the closest point on
the boundaries.
\item Return $\sgn(\inner{x'-x}{n_{x'}})\norm{x'-x}$ as the SDF value at
$x$, where $n_{x'}$ denotes the outward-pointing normal vector of the
boundary at $x'$, and $\sgn(\cdot)$ represents the sign function.
\end{enumerate}

\begin{remark}
The necessity of Newton's method in this algorithm depends on the accuracy 
needed for the SDF evaluation. Without the use of Newton's method, the
accuracy is proportional to $h$, where $h$ is the spacing of the
equidistant sampling points.
\end{remark}

\begin{remark}
One can also construct a signed distance function from a domain
represented as an implicit function (see Section 4 in \cite{distmesh}).
\end{remark}

\subsection{Distmesh}
  \label{sec:distmesh}
Distmesh \cite{distmesh} is a simple and short algorithm for generating
a triangle mesh for a geometry with boundaries described by a signed distance
function. The algorithm is based on the physical analogy between a simplex mesh
and a truss structure, and the triangle mesh is computed by 
solving an ordinary differential equation for the equilibrium in the truss
structure composed of compressible springs (using piecewise linear
force-displacement relations).
Despite its simplicity, the generated triangles are of high quality,
in the sense that most of the triangles are close to equilateral, which is
important in many applications.

Below, we describe the Distmesh algorithm briefly (see \cite{distmesh} for details).

\begin{enumerate}
\item (Initial guess) Create an initial node distribution arranged in a
triangular tiling inside a bounding box of the input geometry, and
remove the nodes outside the geometry. 
\item (Reset topology) Compute a triangle mesh by applying
the Delauney triangulation algorithm to the nodes. 
\item (Compute equilibrium) Solve for the equilibrium in the truss
structure represented by the triangle mesh using the forward Euler method.
When some nodes end up outside the geometry, move them back to the closest 
point on the boundary, where the closest point is computed using the SDF.
\item Go back to Step 2 when the total movements of nodes in Step 3 become large. 
Otherwise, return the final triangle mesh.
\end{enumerate}

\begin{observation}
  \label{obs:adap_mesh}
It is often desirable to resolve a complex geometry by adaptive meshing, i.e.,
requiring mesh elements close to a singularity to be smaller relative to
elements further away. This can be achieved using Distmesh by setting the
equilibrium lengths of the springs (by analogy with a truss system) near the
singularities to be shorter than the equilibrium length of the rest of the
springs. In practice, the desired equilibrium lengths of the springs
are specified by a so-called element size function (see \cite{distmesh}
and Observation \ref{obs:fh} for details).
\end{observation}

\begin{observation}
The quality of a mesh generated by the Distmesh algorithm is relatively
insensitive to the computational accuracy of the solution to the evolution
equations of the truss system, and to the computational accuracy of the SDF
evaluation.  Therefore, the use of the forward Euler method, whose order of
convergence is one, is sufficient for our purpose. For the same reason, when
the SDF is computed using the algorithm described in Appendix
\ref{sec:sdf}, it is unnecessary to use Newton's method to improve its
accuracy.
\end{observation}

\subsection{Modified Distmesh algorithm for constructing a curved
triangulation}
  \label{sec:mod_distmesh}
In Appendix \ref{sec:distmesh}, we describe a simple triangle
mesh generation algorithm named Distmesh, introduced in \cite{distmesh}. 
To apply Distmesh to our particular problem, several additional pieces of
information must be returned by the algorithm, including:
\begin{itemize}
\item Given a mesh element, whether it is a triangle or a curved element;
\item Given a curved element, which vertex is opposite
to that curved side;
\item The curve parameters that correspond to the endpoints of a
curved side.
\end{itemize}
These issues can be remedied with several modifications to the Distmesh
algorithm. In addition, we present a simple technique for improving the
triangle mesh quality using the new outputs.

Firstly, given a triangle mesh, based on the observation that the
boundary edges are the edges that are associated only with a single element,
it is easy to quickly identify all of the boundary edges, boundary vertices
and curved elements (including the vertices opposite to the curved
sides). In addition, we also observe that the use of the signed distance
function (described in Appendix \ref{sec:sdf}) allows us to quickly determine
the curve parameters corresponding to the endpoints of a curved side. So
far, all of the required information stated above can be computed for any
given triangle mesh and its associated boundary.

The mesh quality can be improved by the following modification. 
During the computation of the equilibrium state of the truss structure, as noted
in Appendix \ref{sec:distmesh}, some nodes, especially the ones that are
located on the boundary, may end up outside the geometry. The original
Distmesh algorithm handles this by projecting each outside node to the
closest point on the boundary at the end of each iteration.
In other words, when solving for the
equilibrium, the boundary is neglected until the very end of each
iteration. Such treatment is somehow non-physical, since the boundary should
be treated as a hard obstacle at all stages. Therefore, in our
implementation, we maintain an array of flags that indicates whether or not a node
is on the boundary, and ensure that for every boundary node, the force
component that is normal to the boundary is eliminated. This way, only
forces tangential to the boundary can contribute to the computation.
Although some nodes will still end up outside the geometry and require 
projection back to the closest point on the boundary at the end of each
iteration, their movements become more physical, and it improves the
triangle mesh quality.

Finally, we note that such modification is impossible without keeping track of
the boundary nodes.

\begin{observation}
\label{obs:curva_fh}
Recall that Distmesh uses an element size function to control the sizes
of mesh elements (see Observation \ref{obs:adap_mesh}). We find the element
size function
  \begin{align}
h(x,y):=\oint_{\partial\Omega}e^{-a((x-x_0)^2+(y-y_0)^2)}\cdot\kappa(x_0,y_0)
\d s
  \end{align}
gives a mesh that resolves the boundary of the domain well (see Figure
\ref{fig:wob_domain_mesh}), where $\kappa(x_0,y_0)$ denotes the curvature of
$\partial\Omega$ at $(x_0,y_0)$, and $a$ is a constant that depends on the
scale of the domain and the size of the mesh elements.  
\label{obs:fh}
\end{observation}

\begin{remark}
In fact, the volume potential evaluation algorithm of the paper does not
depend on any particular meshing algorithm. All that is needed is a triangle
mesh of good quality which contains all of the required information
described in this section. 
\end{remark}

\subsection{Nearby mesh elements query}
  \label{sec:query}
    
Due to the singularity of the Green's function, it is important to identify all
nearby mesh elements at a given target location, so that the elements close to
or containing the target can be handled separately from the mesh elements
that are far away.  In this section, we introduce an 
algorithm for finding all of the mesh elements that are within
a certain distance of a given target $x$, together with the particular element
that $x$ lies within (if such an element exist).  The procedure is outlined
as follows. We first preprocess the triangle mesh:

\begin{enumerate}
\item (Sampling) For each mesh element, sample the four vertices of its
bounding box.  Associate each sampling point with the index of the mesh
element that it is sampled from.
\item (Quadtree) Create a quadtree data structure for the sampling points
(see Appendix \ref{sec:quadtree}).
\end{enumerate}

Then, we find the nearby mesh elements of a given target $x$.
Formally speaking, given a query box centered at $x$ of size larger
than the sizes of the bounding boxes of adjacent mesh elements of
the target, we find all of the elements that intersect the query box,
as follows.
\begin{enumerate}
\item Find all the sampling points that are inside the query box by
exploiting the quadtree data structure. 
\item Return the associated element indices of the captured sampling
points.
\end{enumerate}

Finally, we are also able to find the particular mesh element that $x$ lies
within (if any):
\begin{enumerate}
\item (Optional) Evaluate the signed distance function (see Appendix \ref{sec:sdf})
at $x$, and return null if $x$ is outside the domain $\Omega$.
\item Loop through all of the nearby elements of $x$ (obtained through the
previous computation). For each element: 
\begin{itemize}
\item If the element is a triangle, return the index if the barycentric coordinates
of $x$ with respect to the element are all between zero and one;
\item If the element is a curved element (denoting the vertex opposite to
the curved side by $O$), return the index if both of the following
conditions are satisfied:
\begin{itemize}
\item The polar angle of $x$ is in
between the polar angles of the two straight sides of the curved element, with
respect to polar coordinates centered at $O$.
\item The distance between $x$ and $O$ is smaller than the distance between
$x'$ and $O$, where $x'$ denotes the point at which the line $x'O$ and the
curved side intersects.

\end{itemize}
\end{itemize}
\end{enumerate}

As is stated in our assumption above, it is important to have the query box be
larger than the sizes of the bounding boxes of adjacent mesh elements of
the target location, since otherwise, the query box can overlap with the
bounding box without capturing any of its vertices, which leads to uncaptured
nearby elements (see Figure \ref{fig:query_box}).

\begin{figure}[h]
    \centering
    \includegraphics[width=0.6\textwidth]{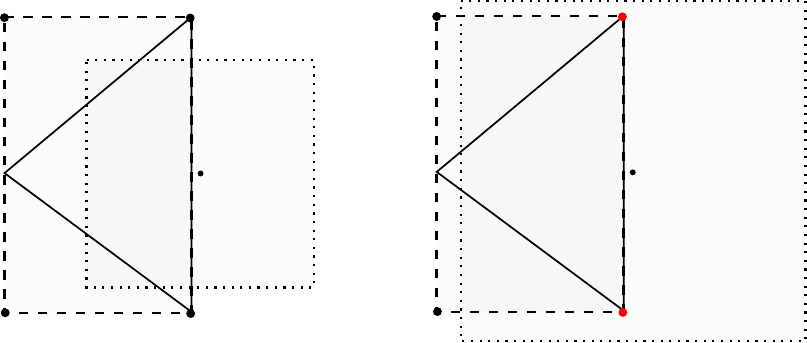}
  \caption{
      {\bf The size of the query box has to be larger than
      its nearby bounding boxes}. On the left, the triangle is not
      captured by the query box, as the size of the box is too small.  This
      can be fixed by increasing the size of the box, as is shown on the
      right.}
   \label{fig:query_box}
\end{figure}

\begin{remark}
  \label{rem:near_query_time}
The construction of the quadtree above takes $\O(n\log n)$ operations, and
the use of the quadtree to query nearby elements takes $\O(\log n+m)$
operations, where $n$ and $m$ are the total number of mesh elements and
nearby elements, respectively (see Appendix \ref{sec:quadtree}). We note that
in practice, $m=\O(1)$.
\end{remark}

\begin{remark}
This query algorithm can be easily generalized to meshes of other types, e.g., 
the quadrilateral mesh.
\end{remark}

\end{document}